\documentclass[a4paper]{amsart}

\usepackage[utf8]{inputenc}
\usepackage[T1]{fontenc}
\usepackage[english]{babel}

\usepackage{amsfonts}
\usepackage{amsmath}
\usepackage{amsrefs}
\usepackage{amssymb}
\usepackage{amstext}
\usepackage{amsthm}

\usepackage{geometry}
\usepackage{fullpage}

\usepackage[shortlabels]{enumitem}
\usepackage{graphicx}
\usepackage{hyperref}
\usepackage{listings}
\usepackage{mathrsfs}
\usepackage{mathtools}
\usepackage{multicol}
\usepackage{shuffle}

\usepackage{tikz}
\usepackage{tikz-cd}
\usetikzlibrary{arrows,backgrounds,chains,decorations.pathreplacing,patterns,shapes}

\newcommand{\area}{\mathsf{area}}
\newcommand{\dinv}{\mathsf{dinv}}
\newcommand{\newdinv}{\mathsf{newdinv}}

\newcommand{\Rise}{\mathsf{Rise}}

\newcommand{\D}{\mathsf{D}} 
\newcommand{\LD}{\mathsf{LD}} 

\newcommand{\PF}{\mathsf{PF}} 
\newcommand{\PLD}{\mathsf{PLD}} 

\newcommand{\RP}{\mathsf{RP}} 

\newcommand{\qbinom}[2]{\genfrac{[}{]}{0pt}{}{#1}{#2}}

\makeatletter

\pgfkeys{
	/tikz/sharp angle/.code={%
		\pgfsetarrowoptions{sharp >}{#1}%
		\pgfsetarrowoptions{sharp <}{-#1}%
	},
	/tikz/sharp > angle/.code={%
		\pgfsetarrowoptions{sharp >}{#1}%
	},
	/tikz/sharp < angle/.code={%
		\pgfsetarrowoptions{sharp <}{#1}%
	},
	/tikz/sharp protrude/.code=\csname if#1\endcsname\qrr@tikz@sharp@z@-0.05\p@\else\qrr@tikz@sharp@z@\z@\fi,
	/tikz/sharp protrude/.default=true
}

\newdimen\qrr@tikz@sharp@z@
\qrr@tikz@sharp@z@\z@
\pgfarrowsdeclare{sharp >}{sharp >}{%
	\edef\pgf@marshal{\noexpand\pgfutil@in@{and}{\pgfgetarrowoptions{sharp >}}}%
	\pgf@marshal
	\ifpgfutil@in@
	\edef\pgf@tempa{\pgfgetarrowoptions{sharp >}}
	\expandafter\qrr@tikz@sharp@parse\pgf@tempa\@qrr@tikz@sharp@parse
	\else
	\qrr@tikz@sharp@parse\pgfgetarrowoptions{sharp >}and-\pgfgetarrowoptions{sharp >}\@qrr@tikz@sharp@parse
	\fi
	\pgfmathparse{max(\pgf@tempa,\pgf@tempb,0)}%
	\let\qrr@tikz@sharp@max\pgfmathresult
	\pgfmathsetlength\pgf@xa{.5*\pgflinewidth * tan(\qrr@tikz@sharp@max)}%
	\pgfarrowsleftextend{+\pgf@xa}%
	\pgfarrowsrightextend{+\pgf@xa}%
}{%
	\edef\pgf@marshal{\noexpand\pgfutil@in@{and}{\pgfgetarrowoptions{sharp >}}}%
	\pgf@marshal
	\ifpgfutil@in@
	\edef\pgf@tempa{\pgfgetarrowoptions{sharp >}}
	\expandafter\qrr@tikz@sharp@parse\pgf@tempa\@qrr@tikz@sharp@parse
	\else
	\qrr@tikz@sharp@parse\pgfgetarrowoptions{sharp >}and-\pgfgetarrowoptions{sharp >}\@qrr@tikz@sharp@parse
	\fi
	\pgfmathsetlength\pgf@ya{.5*\pgflinewidth * tan(max(\pgf@tempa,\pgf@tempb,0))}%
	\pgfmathsetlength\pgf@xa{-.5*\pgflinewidth * tan(\pgf@tempa)}%
	\pgfmathsetlength\pgf@xb{-.5*\pgflinewidth * tan(\pgf@tempb)}%
	\advance\pgf@xa\pgf@ya
	\advance\pgf@xb\pgf@ya
	\ifdim\pgf@xa>\pgf@xb
	\pgftransformyscale{-1}%
	\pgf@xc\pgf@xb
	\pgf@xb\pgf@xa
	\pgf@xa\pgf@xc
	\fi
	\pgfpathmoveto{\pgfqpoint{\qrr@tikz@sharp@z@}{.5\pgflinewidth}}%
	\pgfpathlineto{\pgfqpoint{\pgf@xa}{.5\pgflinewidth}}%
	\pgfpathlineto{\pgfqpoint{\pgf@ya}{+0pt}}%
	\pgfpathlineto{\pgfqpoint{\pgf@xb}{-.5\pgflinewidth}}%
	\pgfpathlineto{\pgfqpoint{\qrr@tikz@sharp@z@}{-.5\pgflinewidth}}%
	\pgfusepathqfill
}
\pgfarrowsdeclare{sharp <}{sharp <}{%
	\edef\pgf@marshal{\noexpand\pgfutil@in@{and}{\pgfgetarrowoptions{sharp <}}}%
	\pgf@marshal
	\ifpgfutil@in@
	\edef\pgf@tempa{\pgfgetarrowoptions{sharp <}}
	\expandafter\qrr@tikz@sharp@parse\pgf@tempa\@qrr@tikz@sharp@parse
	\else
	\expandafter\qrr@tikz@sharp@parse\pgfgetarrowoptions{sharp <}and-\pgfgetarrowoptions{sharp <}\@qrr@tikz@sharp@parse
	\fi
	\pgfmathparse{max(\pgf@tempa,\pgf@tempb,0)}%
	\let\qrr@tikz@sharp@max\pgfmathresult
	\pgfmathsetlength\pgf@xa{.5*\pgflinewidth * tan(\qrr@tikz@sharp@max)}%
	\pgfarrowsleftextend{+\pgf@xa}%
	\pgfarrowsrightextend{+\pgf@xa}%
}{%
	\edef\pgf@marshal{\noexpand\pgfutil@in@{and}{\pgfgetarrowoptions{sharp <}}}%
	\pgf@marshal
	\ifpgfutil@in@
	\edef\pgf@tempa{\pgfgetarrowoptions{sharp <}}
	\expandafter\qrr@tikz@sharp@parse\pgf@tempa\@qrr@tikz@sharp@parse
	\else
	\expandafter\qrr@tikz@sharp@parse\pgfgetarrowoptions{sharp <}and-\pgfgetarrowoptions{sharp <}\@qrr@tikz@sharp@parse
	\fi
	\pgfmathsetlength\pgf@ya{.5*\pgflinewidth * tan(max(\pgf@tempa,\pgf@tempb,0))}%
	\pgfmathsetlength\pgf@xa{-.5*\pgflinewidth * tan(\pgf@tempa)}%
	\pgfmathsetlength\pgf@xb{-.5*\pgflinewidth * tan(\pgf@tempb)}%
	\advance\pgf@xa\pgf@ya
	\advance\pgf@xb\pgf@ya
	\ifdim\pgf@xa>\pgf@xb
	\pgftransformyscale{-1}%
	\pgf@xc\pgf@xb
	\pgf@xb\pgf@xa
	\pgf@xa\pgf@xc
	\fi
	\pgfpathmoveto{\pgfqpoint{\qrr@tikz@sharp@z@}{.5\pgflinewidth}}%
	\pgfpathlineto{\pgfqpoint{\pgf@xa}{.5\pgflinewidth}}%
	\pgfpathlineto{\pgfqpoint{\pgf@ya}{+0pt}}%
	\pgfpathlineto{\pgfqpoint{\pgf@xb}{-.5\pgflinewidth}}%
	\pgfpathlineto{\pgfqpoint{\qrr@tikz@sharp@z@}{-.5\pgflinewidth}}%
	\pgfusepathqfill
}
\def\qrr@tikz@sharp@parse#1and#2\@qrr@tikz@sharp@parse{\def\pgf@tempa{#1}\def\pgf@tempb{#2}}

\makeatother

\newcommand\multiset[2]%
{\mathchoice{\left(\kern-0.4em{\binom{#1}{#2}}\kern-0.4em\right)}
	{\bigl(\kern-0.2em{\binom{#1}{#2}}\kern-0.2em\bigr)}
	{\bigl(\kern-0.2em{\binom{#1}{#2}}\kern-0.2em\bigr)}
	{\bigl(\kern-0.2em{\binom{#1}{#2}}\kern-0.2em\bigr)}}

\let\existstemp\exists \renewcommand*{\exists}{\mathop \existstemp}
\let\foralltemp\forall \renewcommand*{\forall}{\mathop \foralltemp}

\def\quotient#1#2{\raise1ex\hbox{$#1$}\Big/\lower1ex\hbox{$#2$}}

\newcommand{\<}{\langle}
\renewcommand{\>}{\rangle}

\newtheorem{theorem}{Theorem}[section]

\newtheorem{proposition}[theorem]{Proposition}
\newtheorem{corollary}[theorem]{Corollary}
\newtheorem{conjecture}[theorem]{Conjecture}

\theoremstyle{definition}
\newtheorem{definition}[theorem]{Definition}

\theoremstyle{remark}
\newtheorem{remark}[theorem]{Remark}

\title{The new dinv is not so new}

\author{Michele D'Adderio}
\address{Universit\'e Libre de Bruxelles (ULB)\\D\'epartement de Math\'ematique\\ Boulevard du Triomphe, B-1050 Bruxelles\\ Belgium}\email{mdadderi@ulb.ac.be}

\author{Alessandro Iraci}
\address{Universit\'a di Pisa and Universit\'e Libre de Bruxelles (ULB)\\Dipartimento di Matematica\\ Largo Bruno Pontecorvo 5, 56127 Pisa\\ Italia}\email{iraci@student.dm.unipi.it}

\begin{document}
	
\begin{abstract}
	In \cite{Duane-Garsia-Zabrocki-2013} the authors introduced a new dinv statistic, denoted \emph{ndinv}, on the two part case of the shuffle conjecture \cite{HHLRU-2005} in order to prove a compositional refinement. Though in \cite{Hicks-Kim_ndinv} a non-recursive (but algorithmic) definition of ndinv has been given, this statistic still looks a bit unnatural. In this paper we ``unveil the mystery'' around the ndinv, by showing bijectively that the ndinv actually matches the usual dinv statistic in a special case of the generalized Delta conjecture in \cite{Haglund-Remmel-Wilson-2015}. Moreover, we give also a non-compositional proof of the ``$ehh$'' case of the shuffle conjecture (after \cite{Garsia-Xin-Zabrocki-threeshuffle}) by bijectively proving a relation with the two part case of the Delta conjecture.
\end{abstract}

\dedicatory{Dedicated to Adriano Garsia on the occasion of his 90th birthday}
	
\maketitle
\tableofcontents

\section{Introduction}

In \cite{HHLRU-2005} the authors proposed a combinatorial formula for the symmetric function $\Delta_{e_{n-1}}'e_n$, which was known to give the Frobenius characteristic of the so called diagonal harmonics (cf. \cite{Haiman-Vanishing-2002}). This went under the name of \emph{shuffle conjecture} and it kept busy several researchers for quite sometime. After proving some important special cases, like the famous \emph{$q,t$-Catalan} by Garsia and Haglund \cite{Garsia-Haglund-qtCatalan-2002} or the more general \emph{$q,t$-Schr\"{o}der} by Haglund \cite{Haglund-Schroeder-2004}, not much progress had been made for a few years.

In \cite{Haglund-Morse-Zabrocki-2012} Haglund, Morse and Zabrocki formulated a compositional refinement of the shuffle conjecture, which gave a new impulse to the search for a proof of this intriguing conjecture. 

Looking at special cases of the \emph{compositional shuffle conjecture}, in \cite{Duane-Garsia-Zabrocki-2013} Duane, Garsia and Zabrocki introduced a new dinv statistic, which they denoted \emph{ndinv}, to provide a compositional refinement of the two part shuffle conjecture. This statistic looked a bit unnatural, as it was originally defined recursively. Later in \cite{Hicks-Kim_ndinv} Hicks and Kim were able to give a non-recursive (though algorithmic) definition of ndinv: still the notion appears quite artificial.

In this work we ``unveil the mystery'' around this ndinv statistic by establishing a relation between the two part shuffle conjecture and a generalization of the shuffle conjecture due to Haglund, Remmel and Wilson \cite{Haglund-Remmel-Wilson-2015}, known as \emph{(generalized) Delta conjecture}. In fact we show how the ndinv is none other than the natural dinv statistic, but read on the appropriate subset of partially labelled Dyck paths. We do this by combining two bijections in \cite{DAdderio-Iraci-VandenWyngaerd-GenDeltaSchroeder} and \cite{DAdderio-Iraci-polyominoes-2017} that involve parallelogram polyominoes.

In the same line of research, Garsia, Xin and Zabrocki in \cite{Garsia-Xin-Zabrocki-threeshuffle} managed to prove the case $\<\cdot ,e_ah_bh_{n-a-b}\>$ of the compositional shuffle conjecture: this was the most general special case of the shuffle conjecture that had been proved before the breakthrough proof of the full compositional shuffle conjecture by Carlsson and Mellit in \cite{Carlsson-Mellit-ShuffleConj-2015}. 

In the present article we establish a bijective relation between the case $\<\cdot ,e_ah_bh_{n-a-b}\>$ of the shuffle conjecture and the case $\<\cdot ,h_dh_{n-d}\>$ of the Delta conjecture in \cite{Haglund-Remmel-Wilson-2015}. This will allow us to give an alternative proof of the case $\<\cdot ,e_ah_bh_{n-a-b}\>$ of the shuffle conjecture, independent of its compositional refinement. It is worth noticing that, in the current absence of a compositional refinement of the Delta conjecture, our new proof gives at least a hope that it might be possible to prove the same case of the Delta conjecture (so far still open) with the available methods.

\medskip

The rest of the article is organized in the following way. In Section~2 we give the combinatorial definitions that we need in the article. In Section~3 we state the generalized Delta conjecture, while in Section~4 we introduce the missing definitions from symmetric function theory and we prove the identities needed later. In Section~5 we discuss the ndinv and we show how it matches the dinv on certain partially labelled Dyck paths. Finally, in Section~6 we give our non-compositional proof of the ``$ehh$'' case of the shuffle conjecture by establishing a relation with the two part Delta conjecture.

\section{Combinatorial definitions}

In this section we introduce some of the basic combinatorial definitions that we are going to use in the rest of the article.

\subsection{Dyck paths}

\begin{definition}
	A \emph{Dyck path} of size $n$ is a lattice path going from $(0,0)$ to $(n,n)$, using only North and East unit steps and staying weakly above the line $x=y$ called the \emph{main diagonal}. A \emph{labelled Dyck path} is a Dyck path whose vertical steps are labelled with (not necessarily distinct) positive integers such that the labels appearing in each column are strictly increasing from bottom to top. See Figure~\ref{fig:LDP} for an example. A \emph{partially labelled Dyck path} is a labelled Dyck path in which we allow $0$ to be a label, except for the first North step (the bottom-left one).
\end{definition}

The set of all Dyck paths of size $n$ is denoted by $\D(n)$, the set of labelled Dyck paths of size $n$ is denoted by $\LD(n)$, and the set of partially labelled Dyck paths of size $m+n$ with $m$ zero labels is denoted by $\PLD(m,n)$. We identify $\LD(n)$ with $\PLD(0,n)$. For $D \in \PLD(m,n)$ we set $l_i(D)$ to be the label of the $i$-th vertical step.

\begin{figure}[!ht]
	\centering
	
	\begin{tikzpicture}[scale=0.6]
	\draw[gray!60, thin] (0,0) grid (8,8) (0,0) -- (8,8);
	\draw[blue!60, line width = 1.6pt] (0,0)|-(2,3)|-(5,5)|-(6,7)|-(8,8);
	
	\draw
	(0.5,0.5) circle(0.4 cm) node {$2$}
	(0.5,1.5) circle(0.4 cm) node {$4$}
	(0.5,2.5) circle(0.4 cm) node {$5$}
	(2.5,3.5) circle(0.4 cm) node {$1$}
	(2.5,4.5) circle(0.4 cm) node {$3$}
	(5.5,5.5) circle(0.4 cm) node {$2$}
	(5.5,6.5) circle(0.4 cm) node {$6$}
	(6.5,7.5) circle(0.4 cm) node {$1$};
	\end{tikzpicture}
	
	\caption{A labelled Dyck path.}
	\label{fig:LDP}
\end{figure}
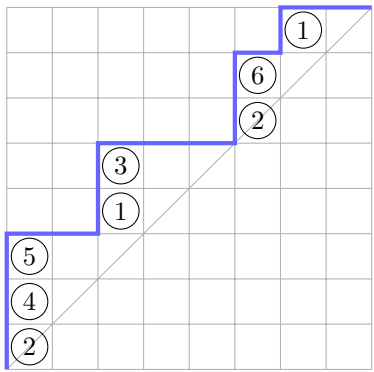

\begin{definition}
	A \emph{parking function} of size $n$ is a function $f \colon [n] \rightarrow [n]$ such that $ \# \{ 1 \leq j \leq n \mid f(j) \geq i \} \leq n+1-i$ (here we used the standard notation $[n] \coloneqq \{1,2,\dots,n\}$).
\end{definition}

We denote by $\PF(n)$ the set of all the parking functions of size $n$. These are in bijective correspondence with the subset of labelled Dyck paths of size $n$ whose labels are exactly the numbers from $1$ to $n$. In particular, a parking function $f$ corresponds to the labelled Dyck path with label $i$ in column $f(i)$ for $1 \leq i \leq n$. We will usually the identify parking functions with the corresponding labelled Dyck paths.

\begin{definition}
	An \emph{area word} is a (finite) string of symbols $a_1 a_2 \cdots a_n$ in a well-ordered alphabet such that if $a_i < a_{i+1}$ then $a_{i+1}$ is the successor of $a_i$ in the alphabet.
\end{definition}

\begin{definition}
	Let $D$ be a (partially labelled) Dyck path of size $n$. We define its \emph{area word} to be the string of integers $a(D) = a_1(D)a_2(D) \cdots a_n(D)$ where $a_i(D)$ is the number of whole squares in the $i$-th row (from the bottom) between the path and the main diagonal.
\end{definition}

Notice that the area word of a Dyck path is an area word in the alphabet $\mathbb{N}$.

\begin{definition}
	We define the statistic \emph{area} on $\D(m+n)$ and $\PLD(m,n)$ as \[ \area(D) \coloneqq \sum_{i=1}^{m+n} a_i(D). \]
\end{definition}

For example, the area word of the path in Figure \ref{fig:LDP} is $01212011$, so its area is $8$.

\begin{definition}\label{def: inversion LD}
	Let $D \in \PLD(m,n)$. For $1 \leq i < j \leq m+n$, we say that the pair $(i,j)$ is a \emph{diagonal inversion} if
	\begin{itemize}
		\item either $a_i(D) = a_j(D)$ and $l_i(D) < l_j(D)$ (\emph{primary inversion}),
		\item or $a_i(D) = a_j(D) + 1$ and $l_i(D) > l_j(D)$ (\emph{secondary inversion}).
	\end{itemize}
	Then we define $d_i \coloneqq \# \{ i < j \leq m+n \mid (i,j) \; \text{is a diagonal inversion} \}$.
\end{definition}

\begin{definition}\ \label{def: dinv LD}
	We define the statistic \emph{dinv} on $\PLD(m,n)$ as \[ \dinv(D) \coloneqq \sum_{i=1}^{m+n} d_i(D), \] and on $\D(m+n)$ by assuming that the inequality condition on the labels always holds.
\end{definition}

The number of primary and secondary diagonal inversions are referred to as \emph{primary} and \emph{secondary dinv} respectively. The labelled Dyck path in Figure~\ref{fig:LDP} has dinv equal to $6$: its diagonal inversions are $(2,7), (4,7)$ (primary), and $(2,6), (3,4), (3,8), (5,8)$ (secondary).

\begin{definition} \label{def: dinv reading word LD}
	Let $D \in \PLD(m,n)$. We define its \emph{dinv reading word} as the sequence of the positive labels read starting from the ones in the main diagonal going bottom to top, left to right; next the ones in the diagonal $y=x+1$ bottom to top, left to right; then the ones in the diagonal $y=x+2$ and so on. Notice that some authors use the reverse convention, getting the reverse of our word.
\end{definition}

For example, the labelled Dyck path in Figure~\ref{fig:LDP} has dinv reading word $22416153$.

To each (partially) labelled Dyck path we associate a monomial in the variables $x_1,x_2,\dots$: for $D \in \PLD(m,n)$ we set \[ x^D \coloneqq \prod_{i=1}^{m+n} x_{l_i(D)} \] where we set $x_0 = 1$ (which explains the word \emph{partially}).

These definitions can be extended to decorated Dyck paths. The idea of decorating rises first appeared in \cite{Haglund-Remmel-Wilson-2015}, together with its modification to the area (cf. also \cite{Zabrocki-4Catalan-2016}).

\begin{definition}
	The \emph{rises} of a Dyck path $D$ of size $n$ are the indices \[ \Rise(D) \coloneqq \{2\leq i \leq n\mid a_{i}(D)>a_{i-1}(D)\},\] or, in words, the vertical steps that are directly preceded by another vertical step.
\end{definition}

\begin{definition}
	A \emph{decorated Dyck path} is a Dyck path where certain rises are decorated with a $\ast$. By $\D(m+n)^{\ast k}$ (resp. $\PLD(m,n)^{\ast k}$) we denote the set of Dyck paths (resp. partially labelled Dyck paths) of size $m+n$ with $k$ decorated rises.
\end{definition}

Decorations on the rises influence the area in the following way.

\begin{definition}
	Let $D$ be a (labelled) decorated Dyck path, and let $a_1(D) \cdots a_{m+n}(D)$ be its area word. Now, let $\mathsf{DRise}(D) \subseteq \Rise(D)$ be the set of indices such that $i \in \mathsf{DRise}(D)$ if the $i$-th vertical step of $D$ is a decorated rise. We define the \emph{area} of $D$ as \[ \area(D) \coloneqq \sum_{i \not \in \mathsf{DRise(D)}} a_{i}(D). \]
	
	For a more visual definition, the area is the number of whole squares that lie between the path and the main diagonal, except for the ones in the rows containing a decorated rise.
	
	The \emph{dinv} of a decorated (partially labelled) Dyck path is defined as the dinv of the path ignoring the decorations.
\end{definition}
In the definition of partially labelled Dyck paths we do not allow a zero in the bottom-left corner; but in Section~\ref{sec:two_bij} we will make use also of objects that allow this. We avoid to introduce a new name/notation for these objects, but we want to define for them the notion of ``zero composition''.
\begin{definition}
	\label{def:zerocomp}
	We define the \emph{zero composition} of a partially labelled Dyck path \emph{with a $0$ label in the bottom-left corner} (cf. Section~\ref{sec:two_bij}) as the composition $\alpha \vDash m$ where $\alpha_i$ is the number of $0$ labels between the $i$-th zero label that lies on the main diagonal (included), and the $(i+1)$-th zero label that lies on the main diagonal (not included) if it exists, or the end if it does not.
\end{definition}

\begin{figure}[!ht]
	\begin{center}
		\begin{tikzpicture}[scale=0.5]
		\draw[step=1.0, gray!60, thin] (0,0) grid (12,12);
		
		\draw[gray!60, thin] (0,0) -- (12,12);
		
		\draw[blue!60, line width=1.6pt] (0,0) -- (0,1) -- (0,2) -- (0,3) -- (1,3) -- (1,4) -- (2,4) -- (2,5) -- (3,5) -- (4,5) -- (5,5) -- (5,6) -- (5,7) -- (5,8) -- (6,8) -- (7,8) -- (8,8) -- (8,9) -- (8,10) -- (9,10) -- (9,11) -- (10,11) -- (11,11) -- (11,12) -- (12,12);
		
		
		
		\draw
		(0.5,0.5) circle (0.4 cm) node {$0$}
		(0.5,1.5) circle (0.4 cm) node {$1$}
		(0.5,2.5) circle (0.4 cm) node {$2$}
		(1.5,3.5) circle (0.4 cm) node {$0$}
		(2.5,4.5) circle (0.4 cm) node {$0$}
		(5.5,5.5) circle (0.4 cm) node {$0$}
		(5.5,6.5) circle (0.4 cm) node {$3$}
		(5.5,7.5) circle (0.4 cm) node {$4$}
		(8.5,8.5) circle (0.4 cm) node {$0$}
		(8.5,9.5) circle (0.4 cm) node {$5$}
		(9.5,10.5) circle (0.4 cm) node {$0$}
		(11.5,11.5) circle (0.4 cm) node {$0$};
		\end{tikzpicture}
	\end{center}
	
	\caption{A partially labelled Dyck path with zero composition $\alpha = (3,1,2,1)$.}
	\label{fig:zerocomp}
\end{figure}
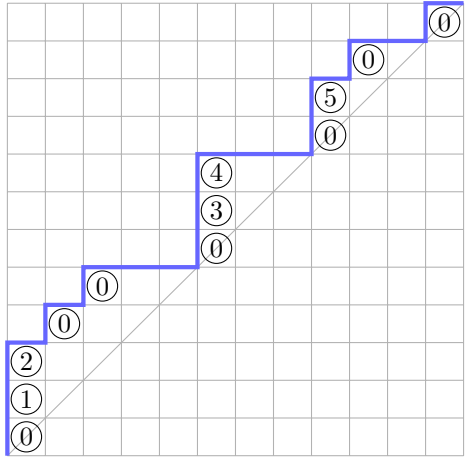

\subsection{Polyominoes}

\begin{definition}
	\label{def:reducedpolyominoes}
	A \emph{reduced (parallelogram) polyomino} of size $m \times n$ is a pair of lattice paths from $(0,0)$ to $(m,n)$ using only north and east steps, such that the first one (the \emph{red path}) always lies weakly above the second one (the \emph{green path}). For an example, see Figure~\ref{fig:plbounce} on the right, or even Figure~\ref{fig:aw-reduced-polyomino}, but ignoring the added notation.
\end{definition}

The set of reduced polyominoes of size $m \times n$ is denoted by $\RP(m,n)$. It is convenient for the purpose of this paper to declare that a reduced polyomino also has a pair of overlapping horizontal steps (one red and one green) from $(-1,0)$ to $(0,0)$. As we will show, this does not alter any of the statistics, but it makes easier to describe certain bijections. We will refer to these steps as \emph{ghost steps}.

Reduced polyominoes are also encoded by their area word, this time in the ordered alphabet $\overline{\mathbb{N}} \coloneqq 0 < \bar{0} < 1 < \bar{1} < 2 < \dots$. The \emph{area word} of a polyomino is computed with a slight modification of the algorithm described in \cite{Aval-DAdderio-Dukes-Hicks-LeBorgne-2014}*{Section~2}. It consists of drawing a diagonal of slope $-1$ in the polyomino from the end of every horizontal green step, and attaching to that step the ``length'' of that diagonal, i.e. the number of unit squares that it crosses. Then, one puts a dot in every square not crossed by any of those diagonals, and attaches to each vertical red step the number of dots in the corresponding row. Next, one bars the numbers attached to vertical red steps, and finally one reads those numbers following the diagonals of slope $-1$, reading the labels when encountering the end of its step and the red label before the green one. See Figure~\ref{fig:aw-reduced-polyomino} for an example. Notice that some diagonals can have length $0$, and that the ghost steps (outside the grid) force the area word to start with a $0$.

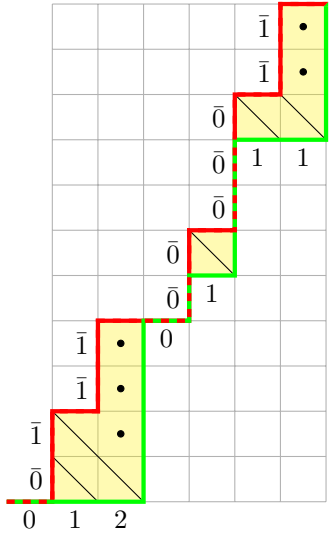
\begin{figure}[!ht]
	\centering
	\begin{tikzpicture}[scale=0.6]
	\draw[gray!60, thin] (0,0) grid (6,11);
	
	\filldraw[yellow, opacity=0.3] (0,0) -- (1,0) -- (2,0) -- (2,1) -- (2,2) -- (2,3) -- (2,4) -- (3,4) -- (3,5) -- (4,5) -- (4,6) -- (4,7) -- (4,8) -- (5,8) -- (6,8) -- (6,9) -- (6,10) -- (6,11) -- (5,11) -- (5,10) -- (5,9) -- (4,9) -- (4,8) -- (4,7) -- (4,6) -- (3,6) -- (3,5) -- (3,4) -- (2,4) -- (1,4) -- (1,3) -- (1,2) -- (0,2) -- (0,1) -- (0,0);
	
	\draw
	(1,0) -- (0,1)
	(2,0) -- (0,2)
	(4,5) -- (3,6)
	(5,8) -- (4,9)
	(6,8) -- (5,9);
	
	\filldraw[fill=black]
	(1.5,1.5) circle (2pt)
	(1.5,2.5) circle (2pt)
	(1.5,3.5) circle (2pt)
	(5.5,9.5) circle (2pt)
	(5.5,10.5) circle (2pt);
	
	\node[below] at (-0.5,0) {$0$};
	\node[below] at (0.5,0) {$1$};
	\node[below] at (1.5,0) {$2$};
	\node[below] at (2.5,4) {$0$};
	\node[below] at (3.5,5) {$1$};
	\node[below] at (4.5,8) {$1$};
	\node[below] at (5.5,8) {$1$};
	
	\node[left] at (0,0.5) {$\bar{0}$};
	\node[left] at (0,1.5) {$\bar{1}$};
	\node[left] at (1,2.5) {$\bar{1}$};
	\node[left] at (1,3.5) {$\bar{1}$};
	\node[left] at (3,4.5) {$\bar{0}$};
	\node[left] at (3,5.5) {$\bar{0}$};
	\node[left] at (4,6.5) {$\bar{0}$};
	\node[left] at (4,7.5) {$\bar{0}$};
	\node[left] at (4,8.5) {$\bar{0}$};
	\node[left] at (5,9.5) {$\bar{1}$};
	\node[left] at (5,10.5) {$\bar{1}$};
	
	\draw[red, sharp <-sharp >, sharp > angle = -45, line width=1.6pt] (-1,0) -- (0,0) -- (0,1) -- (0,2) -- (1,2) -- (1,3) -- (1,4) -- (2,4) -- (3,4) -- (3,5) -- (3,6) -- (4,6) -- (4,7) -- (4,8) -- (4,9) -- (5,9) -- (5,10) -- (5,11) -- (6,11);
	
	\draw[green, sharp <-sharp >, sharp > angle = 45, line width=1.6pt] (-1,0) -- (0,0) -- (1,0) -- (2,0) -- (2,1) -- (2,2) -- (2,3) -- (2,4) -- (3,4) -- (3,5) -- (4,5) -- (4,6) -- (4,7) -- (4,8) -- (5,8) -- (6,8) -- (6,9) -- (6,10) -- (6,11);
	
	\draw[red, dashed, sharp <-sharp >, sharp > angle = -45, line width=1.6pt] (-1,0) -- (0,0) -- (0,1) -- (0,2) -- (1,2) -- (1,3) -- (1,4) -- (2,4) -- (3,4) -- (3,5) -- (3,6) -- (4,6) -- (4,7) -- (4,8) -- (4,9) -- (5,9) -- (5,10) -- (5,11) -- (6,11);
	\end{tikzpicture}
	
	\caption{A $6 \times 11$ reduced polyomino. Its area word is $0 \bar{0} 1 \bar{1} 2 \bar{1} \bar{1} 1 \bar{0} \bar{0} 0 \bar{0} \bar{0} \bar{0} 1 1 \bar{1} \bar{1}$.}
	\label{fig:aw-reduced-polyomino}
\end{figure}

The proof of the following proposition is identical to the proof of \cite{Aval-DAdderio-Dukes-Hicks-LeBorgne-2014}*{Corollary~3.2}, hence it is omitted.

\begin{proposition}
	For $m, n \geq 0$, there is a bijective correspondence between $m \times n$ reduced polyominoes and area words of length $m+n+1$ in the alphabet $\overline{\mathbb{N}}$ starting with $0$, with exactly $m+1$ unbarred letters, and exactly $n$ barred letters.
\end{proposition}

On reduced polyominoes we have statistics \emph{area} and \emph{dinv} similar to the ones we have on Dyck paths. Let $a_i(P)$ be the $i$-th letter of the area word of a reduced polyomino, (where $a_0 = 0$ is the one associated to the ghost horizontal green step) and let $\vert a_i(P) \vert$ be its value, disregarding the bars.

\begin{definition}
	We define the statistic \emph{area} on $\RP(m,n)$ as \[ \area(P) = \sum_{i=0}^{m+n} \vert a_i(P) \vert, \] which is also the number of whole squares between the two paths.
\end{definition}

\begin{definition}
	Let $P \in \RP(m,n)$. For $0 \leq i < j \leq m+n$, we say that the pair $(i,j)$ is a \textit{diagonal inversion} if $a_i(P)$ is the successor of $a_j(P)$ in the alphabet $\overline{\mathbb{N}}$. Then we define $d_i(P) \coloneqq \# \{ i < j \mid (i,j) \; \text{is a diagonal inversion} \}$. Finally, we define the statistic \emph{dinv} on $\RP(m,n)$ as \[ \dinv(P) = \sum_{i=0}^{m+n} d_i(P). \]
\end{definition}

\begin{definition}
	The \emph{rises} of a reduced polyomino $P$ are the indices \[ \Rise(P) \coloneqq \{1 \leq i \leq n \mid \vert a_{i}(P) \vert > \vert a_{i-1}(P) \vert \},\] which correspond in the figure to the diagonals of slope $-1$ from the endpoint of a vertical red step to the endpoint of a horizontal green step.
\end{definition}

\begin{definition}
	A \emph{decorated reduced polyomino} is a reduced polyomino where certain rises are decorated with a symbol $\ast$.
\end{definition}

By $\RP(m,n)^{\ast k}$ we denote the set of reduced polyominoes of size $m \times n$ with $k$ decorated rises. A decoration on a rise does not affect the dinv of a reduced polyomino, while it affects its area in the same way as it does for Dyck paths, namely in computing the area we have to disregard the letters of the area word whose index is a decorated rise.

\subsection{Two car parking functions}

We need one more combinatorial object, namely two car parking functions. 

\begin{definition}
	A \emph{two car parking function} is a labelled Dyck path such that all the labels have value $1$ or $2$.
\end{definition}

We will refer to the $1$'s as \emph{small cars} and to the $2$'s as \emph{big cars}. Once again, it is convenient to think about them as having $m+1$ big cars instead, one of them in the bottom-left corner (which we call \emph{ghost car} and draw in gray). This doesn't alter any of the statistics.

The set of two car parking functions of size $m+n$ with $n$ $1$'s and $m$ $2$'s, and $k$ decorated rises is denoted by $\PF^2(m,n)^{\ast k}$.
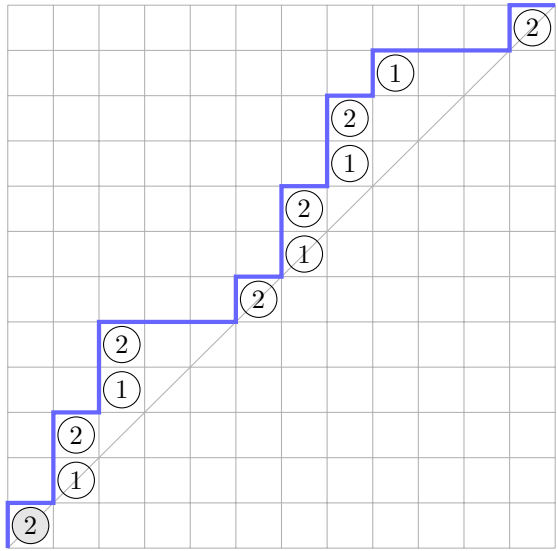
\begin{figure}[!ht]
	\begin{center}
		\begin{tikzpicture}[scale = 0.6]
		\draw[gray!60, thin] (-1,-1) grid (11,11);
		\draw[gray!60, thin] (-1,-1) -- (11,11);
		
		\draw[blue!60, line width=1.6pt] (-1,-1) -- (-1,0) -- (0,0) -- (0,1) -- (0,2) -- (1,2) -- (1,3) -- (1,4) -- (2,4) -- (3,4) -- (4,4) -- (4,5) -- (5,5) -- (5,6) -- (5,7) -- (6,7) -- (6,8) -- (6,9) -- (7,9) -- (7,10) -- (8,10) -- (9,10) -- (10,10) -- (10,11) -- (11,11);
		
		\filldraw[fill=gray!20]
		(-0.5,-0.5) circle (0.4 cm) node {$2$};
		
		\draw
		(0.5,0.5) circle (0.4 cm) node {$1$}
		(0.5,1.5) circle (0.4 cm) node {$2$}
		(1.5,2.5) circle (0.4 cm) node {$1$}
		(1.5,3.5) circle (0.4 cm) node {$2$}
		(4.5,4.5) circle (0.4 cm) node {$2$}
		(5.5,5.5) circle (0.4 cm) node {$1$}
		(5.5,6.5) circle (0.4 cm) node {$2$}
		(6.5,7.5) circle (0.4 cm) node {$1$}
		(6.5,8.5) circle (0.4 cm) node {$2$}
		(7.5,9.5) circle (0.4 cm) node {$1$}
		(10.5,10.5) circle (0.4 cm) node {$2$};
		\end{tikzpicture}
	\end{center}
	
	\caption{A two car parking function with five $1$'s and six $2$'s (plus the ghost car).}
	\label{fig:2cpf}
\end{figure}

The name \emph{two car parking functions} was given because replacing the $1$'s and the $2$'s with two decreasing sequences $n, \dots, 1$ and $m+n, \dots, n+1$ in the dinv reading word, the dinv is not altered. In other words, two car parking functions with $n$ labels equal to $1$ and $m$ labels equal to $2$ behave exactly as parking functions whose dinv reading word is in the shuffle $(n, \dots, 1) \shuffle (m+n, \dots, n+1)$ do, according to the bistatistic $(\dinv, \area)$. To avoid confusion, we will refer to labelled Dyck paths with labels in $\{1,2\}$ as two car parking functions, and to labelled Dyck paths whose reading word is a shuffle of two decreasing sequences as \emph{two shuffle parking functions}.

\begin{definition}
	We define the \emph{big car composition} of $P \in \PF^2(m,n)^{\ast k}$ as the composition $\alpha \vDash m+1$ where $\alpha_i$ is the number of big cars between the $i$-th big car that lies on the main diagonal (included), and the $(i+1)$-th big car that lies on the main diagonal (not included) if it exists, or the end if it does not.
\end{definition}

The two car parking function in Figure~\ref{fig:2cpf} has big car composition $\alpha = (3,3,1)$ (the ghost car is shown in the picture).

\section{The Delta conjecture}

In this short section we limit ourself to recall the definitions needed to state the Delta conjecture, in particular the definition of the Delta operators. For the missing notation we refer to \cite[Section~1]{DAdderio-VandenWyngaerd-2017} or \cite{Haglund-Book-2008} (cf. also Section~\ref{sec:SF}).

We denote by $\Lambda$ the algebra over the field $\mathbb{Q}(q,t)$ of symmetric functions in the variables $x_1,x_2,\dots$. We denote by $e_n$, $h_n$ and $p_n$ the \emph{elementary}, \emph{complete homogeneous} and \emph{power symmetric function} of degree $n$, respectively. We denote by $\omega$ the involution of $\Lambda$ defined by $\omega(e_n) \coloneqq h_n$ for all $n$.

Also, for any partition $\mu$, we denote by $s_{\mu}\in \Lambda$ the corresponding \emph{Schur function}. It is well-known that the symmetric functions $\{s_{\mu}\}_{\mu}$ form a basis of $\Lambda$. The \emph{Hall scalar product} on $\Lambda$, denoted $\<\, ,\, \>$, can be defined by stating that the Schur functions are an orthonormal basis.

Let $\widetilde{H}_{\mu}\in \Lambda$ denote the \emph{(modified) Macdonald polynomial} indexed by the partition $\mu$. As the polynomials $\{\widetilde{H}_{\mu}\}_{\mu}$ form a basis of $\Lambda$, given a symmetric function $f\in \Lambda$, we can define the \emph{Delta operators} $\Delta_f$ and $\Delta_f'$ on $\Lambda$ by setting
\begin{align}
	\Delta_f \widetilde{H}_{\mu} \coloneqq f[B_{\mu}(q,t)] \widetilde{H}_{\mu} \quad \text{ and } \quad 
	\Delta_f' \widetilde{H}_{\mu}  \coloneqq f[B_{\mu}(q,t)-1] \widetilde{H}_{\mu}, \quad \text{ for all } \mu,
\end{align}
where $B_{\mu}(q,t)=\sum_{c\in \mu}q^{a_\mu'(c)}t^{l_\mu'(c)}$, $a_\mu'(c)$ and $l_\mu'(c)$ are the coarm and coleg of $c$ in $\mu$, respectively, and the square brackets denote the \emph{plethystic} substitution.


The first formulation of the generalized Delta conjecture appears in \cite{Haglund-Remmel-Wilson-2015}, which we state in terms of labelled decorated Dyck paths.

\begin{conjecture}[Generalized Delta conjecture] For any integers $n > k \geq 0$, 
	\begin{align}
		\Delta_{h_m} \Delta'_{e_{n-k-1}} e_n =\sum_{D\in \PLD(m,n)^{\ast k}} q^{\dinv(D)}t^{\area(D)} x^D .
	\end{align} 
\end{conjecture}

For $m=0$, this is known as Delta conjecture. For $m=k=0$, the generalized Delta conjecture reduces to the \emph{Shuffle conjecture} in \cite{HHLRU-2005}, recently proved in \cite{Carlsson-Mellit-ShuffleConj-2015}.

\begin{definition}
	Given $\mu \vdash n$, a \emph{$\mu$-shuffle} is a string of numbers from $1$ to $n$ such that the substrings $(1, \dots, \mu_1)$, $(\mu_1+1, \dots, \mu_1+\mu_2)$, $\dots$, $(n-\mu_{\ell(\mu)}+1, \dots, n)$ appear in increasing order. Given $\mu \vdash n-d$ and $\nu \vdash d$, a $\mu,\nu$-shuffle is a string of numbers from $1$ to $n$ such that the substrings $(1, \dots, \mu_1)$,  $\dots$, $(n-\mu_{\ell(\mu)}+1, \dots, n-d)$ appear in increasing order, and the substrings $(n-d+\nu_1, \dots, n-d+1)$, $\dots$, $(n, \dots, n-\nu_{\ell(\nu)}+1)$ appear in decreasing order.
\end{definition}

It is well known (cf. \cite[Chapter~6]{Haglund-Book-2008}) that the generalized Delta conjecture predicts that taking the scalar product of $\Delta'_{e_{n-k-1}}e_n$ with $e_\mu h_\nu$ corresponds to taking the subsets of paths whose dinv reading word is a $\mu,\nu$-shuffle.

In this paper we only deal with the $h_a h_b$ case (i.e. a shuffle of two decreasing sequences) and the $e_a h_b h_c$ case (i.e. a shuffle of an increasing sequence and two decreasing sequences).

\section{Symmetric functions} \label{sec:SF}

In this section we prove the identities of symmetric functions needed in the rest of the article.

\medskip

For all the undefined notations and the unproven identities, we refer to \cite{DAdderio-VandenWyngaerd-2017}*{Section~1}, where definitions, proofs and/or references can be found.

\subsection{Notation}

We denote by $\Lambda=\bigoplus_{n\geq 0}\Lambda^{(n)}$ the graded algebra of symmetric functions with coefficients in $\mathbb{Q}(q,t)$, and by $\<\, , \>$ the \emph{Hall scalar product} on $\Lambda$, which can be defined by saying that the Schur functions form an orthonormal basis.

The standard bases of the symmetric functions that will appear in our
calculations are the monomial $\{m_\lambda\}_{\lambda}$, complete $\{h_{\lambda}\}_{\lambda}$, elementary $\{e_{\lambda}\}_{\lambda}$, power $\{p_{\lambda}\}_{\lambda}$ and Schur $\{s_{\lambda}\}_{\lambda}$ bases.

\medskip 

\emph{We will use implicitly the usual convention that $e_0 = h_0 = 1$ and $e_k = h_k = 0$ for $k < 0$.}

\medskip 

For a partition $\mu\vdash n$, we denote by
\begin{align}
	\widetilde{H}_{\mu} \coloneqq \widetilde{H}_{\mu}[X]=\widetilde{H}_{\mu}[X;q,t]=\sum_{\lambda\vdash n}\widetilde{K}_{\lambda \mu}(q,t)s_{\lambda}
\end{align}
the \emph{(modified) Macdonald polynomials}, where
\begin{align}
	\widetilde{K}_{\lambda \mu} \coloneqq \widetilde{K}_{\lambda \mu}(q,t)=K_{\lambda \mu}(q,1/t)t^{n(\mu)}\quad \text{ with }\quad n(\mu)=\sum_{i\geq 1}\mu_i(i-1)
\end{align}
are the \emph{(modified) Kostka coefficients} (see \cite{Haglund-Book-2008}*{Chapter~2} for more details). 

The set $\{\widetilde{H}_{\mu}[X;q,t]\}_{\mu}$ is a basis of the ring of symmetric functions $\Lambda$. This is a modification of the basis introduced by Macdonald \cite{Macdonald-Book-1995}.

If we identify the partition $\mu$ with its Ferrers diagram, i.e. with the collection of cells $\{(i,j)\mid 1\leq i\leq \mu_i, 1\leq j\leq \ell(\mu)\}$, then for each cell $c\in \mu$ we refer to the \emph{arm}, \emph{leg}, \emph{co-arm} and \emph{co-leg} (denoted respectively as $a_\mu(c), l_\mu(c), a_\mu(c)', l_\mu(c)'$) as the number of cells in $\mu$ that are strictly to the right, above, to the left and below $c$ in $\mu$, respectively.

We set $M \coloneqq (1-q)(1-t)$ and we define for every partition $\mu$
\begin{align}
	B_{\mu} & \coloneqq B_{\mu}(q,t)=\sum_{c\in \mu}q^{a_{\mu}'(c)}t^{l_{\mu}'(c)} \\
	T_{\mu} & \coloneqq T_{\mu}(q,t)=\prod_{c\in \mu}q^{a_{\mu}'(c)}t^{l_{\mu}'(c)} \\
	\Pi_{\mu} & \coloneqq \Pi_{\mu}(q,t)=\prod_{c\in \mu/(1)}(1-q^{a_{\mu}'(c)}t^{l_{\mu}'(c)}) \\
	w_{\mu} & \coloneqq w_{\mu}(q,t)=\prod_{c\in \mu} (q^{a_{\mu}(c)} - t^{l_{\mu}(c) + 1}) (t^{l_{\mu}(c)} - q^{a_{\mu}(c) + 1}).
\end{align}

We will make extensive use of the \emph{plethystic notation} (cf. \cite{Haglund-Book-2008}*{Chapter~1}).

%
%
%

We define the \emph{nabla} operator on $\Lambda$ by
\begin{align}
	\nabla \widetilde{H}_{\mu} \coloneqq T_{\mu} \widetilde{H}_{\mu} \quad \text{ for all } \mu,
\end{align}
and we define the \emph{delta} operators $\Delta_f$ and $\Delta_f'$ on $\Lambda$ by
\begin{align}
	\Delta_f \widetilde{H}_{\mu} \coloneqq f[B_{\mu}(q,t)] \widetilde{H}_{\mu} \quad \text{ and } \quad 
	\Delta_f' \widetilde{H}_{\mu}  \coloneqq f[B_{\mu}(q,t)-1] \widetilde{H}_{\mu}, \quad \text{ for all } \mu.
\end{align}
Observe that on the vector space of symmetric functions homogeneous of degree $n$, denoted by $\Lambda^{(n)}$, the operator $\nabla$ equals $\Delta_{e_n}$. Moreover, for every $1\leq k\leq n$,
\begin{align}
	\label{eq:deltaprime}
	\Delta_{e_k} = \Delta_{e_k}' + \Delta_{e_{k-1}}' \quad \text{ on } \Lambda^{(n)},
\end{align}
and for any $k > n$, $\Delta_{e_k} = \Delta_{e_{k-1}}' = 0$ on $\Lambda^{(n)}$, so that $\Delta_{e_n}=\Delta_{e_{n-1}}'$ on $\Lambda^{(n)}$.

\medskip

For a given $k\geq 1$, we define the \emph{Pieri coefficients} $c_{\mu \nu}^{(k)}$ and $d_{\mu \nu}^{(k)}$ by setting
\begin{align}
	\label{eq:def_cmunu} h_{k}^\perp \widetilde{H}_{\mu}[X] & =\sum_{\nu \subset_k \mu} c_{\mu \nu}^{(k)} \widetilde{H}_{\nu}[X], \\
	\label{eq:def_dmunu} e_{k}\left[\frac{X}{M}\right] \widetilde{H}_{\nu}[X] & = \sum_{\mu \supset_k \nu} d_{\mu \nu}^{(k)} \widetilde{H}_{\mu}[X],
\end{align}
where $h_k^\perp$ is the adjoint operator of the multiplication by $h_k$ with respect to the Hall scalar product, $\nu\subset_k \mu$ means that $\nu$ is contained in $\mu$ (as Ferrers diagrams) and $\mu/\nu$ has $k$ lattice cells, and the symbol $\mu \supset_k \nu$ is analogously defined. The following identity is well-known:
\begin{align}
	\label{eq:rel_cmunu_dmunu} 
	c_{\mu \nu}^{(k)} = \frac{w_{\mu}}{w_{\nu}}d_{\mu \nu}^{(k)}.
\end{align}


\medskip

%

Recall also the standard notation for $q$-analogues: for $n, k\in \mathbb{N}$, we set
\begin{align}
	[0]_q \coloneqq 0, \quad \text{ and } \quad [n]_q \coloneqq \frac{1-q^n}{1-q} = 1+q+q^2+\cdots+q^{n-1} \quad \text{ for } n \geq 1,
\end{align}
\begin{align}
	[0]_q! \coloneqq 1 \quad \text{ and }\quad [n]_q! \coloneqq [n]_q[n-1]_q \cdots [2]_q [1]_q \quad \text{ for } n \geq 1,
\end{align}
and
\begin{align}
	\qbinom{n}{k}_q  \coloneqq \frac{[n]_q!}{[k]_q![n-k]_q!} \quad \text{ for } n \geq k \geq 0, \quad \text{ while } \quad \qbinom{n}{k}_q \coloneqq 0 \quad \text{ for } n < k.
\end{align}

%

\subsection{Some basic identities}

First of all, we record the well-known
\begin{equation} \label{eq:MacHook}
	\<\widetilde{H}_{\mu}, s_{n-r,1^r}\>=e_r[B_\mu-1]\qquad \text{ for all }\mu\vdash n.
\end{equation}

The following identity is also well-known: for any symmetric function $f\in \Lambda^{(n)}$,
\begin{align}
	\label{eq:lem_e_h_Delta}
	\< \Delta_{e_{d}} f, h_n \> = \< f, e_d h_{n-d} \>.
\end{align}

We will use the following form of \emph{Macdonald-Koornwinder reciprocity}: for all partitions $\alpha$ and $\beta$
\begin{align}
	\label{eq:Macdonald_reciprocity}
	\frac{\widetilde{H}_{\alpha}[MB_{\beta}]}{\Pi_{\alpha}} = \frac{\widetilde{H}_{\beta}[MB_{\alpha}]}{\Pi_{\beta}}.
\end{align}
The following identity is also known as \emph{Cauchy identity}:
\begin{align}
	\label{eq:Mac_Cauchy}
	e_n \left[ \frac{XY}{M} \right] = \sum_{\mu \vdash n} \frac{ \widetilde{H}_{\mu} [X] \widetilde{H}_\mu [Y]}{w_\mu} \quad \text{ for all } n.
\end{align}

We need the following well-known proposition.
\begin{proposition} 
	For $n\in \mathbb{N}$ we have
	\begin{align}
		\label{eq:en_expansion}
		e_n[X] = e_n \left[ \frac{XM}{M} \right] = \sum_{\mu \vdash n} \frac{M B_\mu \Pi_{\mu} \widetilde{H}_\mu[X]}{w_\mu}.
	\end{align}
	Moreover, for all $k\in \mathbb{N}$ with $0\leq k\leq n$, we have
	\begin{align}
		\label{eq:e_h_expansion}
		h_k \left[ \frac{X}{M} \right] e_{n-k} \left[ \frac{X}{M} \right] = \sum_{\mu \vdash n} \frac{e_k[B_\mu] \widetilde{H}_\mu[X]}{w_\mu}.
	\end{align}
\end{proposition}

\subsection{A few useful identities}

We recall here \cite{Haglund-Schroeder-2004}*{Theorem~2.6}, i.e. for any $A,F\in \Lambda$ homogeneous
\begin{equation} \label{eq:HaglundThm}
	\sum_{\mu\vdash n}\Pi_\mu F[MB_\mu]d_{\mu\nu}^A=\Pi_\nu\left(\Delta_{A[MX]}F[X]\right)[MB_\nu],
\end{equation}
where $d_{\mu\nu}^A$ is the generalized Pieri coefficient defined by
\begin{equation}
	\sum_{\mu\supset \nu}d_{\mu\nu}^A\widetilde{H}_\mu=A\widetilde{H}_\nu.
\end{equation}

Observe that \eqref{eq:MacHook} implies
\[\< \Delta_{h_{n}}\Delta_{e_{m-k}}'e_{m+1}, h_{m+1} \>=\< \Delta_{h_{n}}e_{m+1}, s_{k+1,1^{m-k}} \>\]

We recall here two theorems from \cite{DAdderio-Iraci-polyominoes-2017}.
\begin{theorem}[\cite{DAdderio-Iraci-polyominoes-2017}*{Theorem~5.1}]
	For $n,m,k\geq 0$ and $m\geq k\geq 0$, we have
	\begin{equation} \label{eq:Deltahh_lem}
		\< \Delta_{e_{m+n-k-1}}'e_{m+n}, h_{m}h_n \>= \< \Delta_{h_{n}}\Delta_{e_{m-k}}'e_{m+1}, h_{m+1} \>.
	\end{equation}
\end{theorem}
\begin{theorem}[\cite{DAdderio-Iraci-polyominoes-2017}*{Theorem~5.3}] \label{thm:delta_hh}
	For $n,m,k\geq 0$ and $m\geq k\geq 0$, we have
	\begin{equation}
		\sum_{r=1}^{m-k+1} t^{m-k-r+1} \< \Delta_{h_{m-k-r+1}} \Delta_{e_k} e_n \left[ X \dfrac{1 - q^r}{1 - q} \right], e_n \>= \< \Delta_{h_{n}}\Delta_{e_{m-k}}'e_{m+1}, h_{m+1} \>.
	\end{equation}
\end{theorem}

Finally, we recall \cite{DAdderio-VandenWyngaerd-2017}*{Lemma~5.2}, i.e. for every $n,k\in \mathbb{N}$, with $n> k\geq 1$, and for every $\beta\vdash n$, we have
\begin{equation} \label{eq:lem_eBcmunu}
	e_{n-k-1}[B_{\beta}-1]B_{\beta}=\sum_{\gamma \subset_k \beta} c_{\beta \gamma}^{(k)}B_{\gamma} T_{\gamma} .
\end{equation}

\subsection{A new identity}

We want to prove the following new identity.
\begin{theorem} \label{thm:new_id}
	For $n,m,k\geq 0$ and $m\geq k\geq 0$, we have
	\begin{equation}
		\< \Delta_{h_{n}}\Delta_{e_{m-k}}'e_{m+1}, h_{m+1} \>  =\< \nabla e_{m+n-k}, e_k h_{n-k} h_{m-k} \>.
	\end{equation}
\end{theorem}
\begin{proof}
	On one hand
	\begin{align*}
		&  \hspace{-0.5cm}\< \nabla e_{m+n-k}, e_k h_{n-k} h_{m-k} \>=  \\
		\text{(using \eqref{eq:en_expansion})}& =\sum_{\mu\vdash m+n-k} M \frac{\Pi_\mu}{w_\mu} T_\mu B_\mu \<\widetilde{H}_\mu,e_k h_{n-k} h_{m-k}\> \\
		& =\sum_{\mu\vdash m+n-k} M \frac{\Pi_\mu}{w_\mu} T_\mu B_\mu \<h_{m-k}^\perp \widetilde{H}_\mu,e_k h_{n-k} \> \\
		\text{(using \eqref{eq:def_cmunu})}& =\sum_{\mu\vdash m+n-k} M \frac{\Pi_\mu}{w_\mu} T_\mu B_\mu \sum_{\beta\subset_{m-k} \mu}c_{\mu\beta}^{(m-k)}\< \widetilde{H}_\beta,e_k h_{n-k} \> \\
		\text{(using \eqref{eq:lem_e_h_Delta})}& =\sum_{\mu\vdash m+n-k} M \frac{\Pi_\mu}{w_\mu} T_\mu B_\mu \sum_{\beta\subset_{m-k} \mu}c_{\mu\beta}^{(m-k)} e_k[B_\beta]  \\
		\text{(using \eqref{eq:rel_cmunu_dmunu})}& =\sum_{\beta\vdash n} M \frac{1}{w_\beta} e_k[B_\beta]\sum_{\mu\supset_{m-k} \beta}  \Pi_\mu  T_\mu B_\mu d_{\mu\beta}^{(m-k)}   \\
		\text{(using \eqref{eq:HaglundThm})}& =\sum_{\beta\vdash n} M \frac{1}{w_\beta} e_k[B_\beta] \Pi_\beta \left.(\Delta_{e_{m-k}}e_1[X/M]e_{m+n-k}[X/M])\right|_{X=MB_\beta}\\
		\text{(using \eqref{eq:e_h_expansion})}& =\sum_{\beta\vdash n} M \frac{\Pi_\beta}{w_\beta} e_k[B_\beta]  \sum_{\lambda\vdash m+n-k+1}e_{m-k}[B_\lambda]B_\lambda\frac{\widetilde{H}_\lambda[MB_\beta]}{w_\lambda}.
	\end{align*}
	On the other hand
	\begin{align*}
		& \hspace{-0.5cm} \< \Delta_{h_{n}}\Delta_{e_{m-k}}'e_{m+1}, h_{m+1} \>= \\
		\text{(using \eqref{eq:en_expansion})}& =\sum_{\lambda \vdash m+1} h_n[B_\lambda ]e_{m-k}[B_\lambda -1] \frac{M B_\lambda \Pi_\lambda }{w_\lambda }\\
		\text{(using \eqref{eq:lem_eBcmunu})}& =\sum_{\lambda \vdash m+1} h_n[B_\beta ]\sum_{\gamma \subset_k \lambda} c_{\lambda\gamma}^{(k)}B_\gamma T_\gamma\frac{M  \Pi_\lambda }{w_\lambda } \\
		\text{(using \eqref{eq:rel_cmunu_dmunu})}& =\sum_{\gamma \vdash m-k+1} \frac{M B_\gamma T_\gamma}{w_\gamma} \sum_{\lambda \supset_k \gamma} h_n[B_\lambda ] d_{\lambda\gamma}^{(k)}  \Pi_\lambda  \\
		\text{(using \eqref{eq:HaglundThm})}& =\sum_{\gamma \vdash m-k+1} \frac{M B_\gamma T_\gamma}{w_\gamma} \Pi_\gamma\left.(\Delta_{e_k}h_n[X/M])\right|_{X=MB_\gamma}\\
		\text{(using \eqref{eq:e_h_expansion})}& =\sum_{\gamma \vdash m-k+1} \frac{M B_\gamma T_\gamma}{w_\gamma} \Pi_\gamma \sum_{\beta \vdash n} e_k[B_\beta]\frac{T_\beta\widetilde{H}_\beta[MB_\gamma]}{w_\beta}\\
		\text{(using \eqref{eq:Macdonald_reciprocity})}& =\sum_{\beta\vdash n} M\frac{\Pi_\beta}{w_\beta}   e_k[B_\beta] T_\beta \sum_{\gamma \vdash m-k+1} \frac{B_\gamma T_\gamma}{w_\gamma} \widetilde{H}_\gamma[MB_\beta] .
	\end{align*}
	So our identity would follow from
	\begin{equation} \label{eq:auxiliar}
		\sum_{\lambda\vdash m+n-k+1}e_{m-k}[B_\lambda]B_\lambda\frac{\widetilde{H}_\lambda[MB_\beta]}{w_\lambda}=T_\beta \sum_{\gamma \vdash m-k+1} \frac{B_\gamma T_\gamma}{w_\gamma} \widetilde{H}_\gamma[MB_\beta].
	\end{equation}
	But 
	\begin{align*}
		& \hspace{-0.5cm} T_\beta \sum_{\gamma \vdash m-k+1} \frac{B_\gamma T_\gamma}{w_\gamma} \widetilde{H}_\gamma[MB_\beta]=\\
		& = \sum_{\gamma \vdash m-k+1} \frac{B_\gamma T_\gamma}{w_\gamma} T_\beta \widetilde{H}_\gamma[MB_\beta] \\
		\text{($\beta\vdash n$)}& = \sum_{\gamma \vdash m-k+1} \frac{B_\gamma T_\gamma}{w_\gamma} e_n[B_\beta] \widetilde{H}_\gamma[MB_\beta]\\
		\text{(using \eqref{eq:def_dmunu})}& = \sum_{\gamma \vdash m-k+1} \frac{B_\gamma T_\gamma}{w_\gamma} \sum_{\lambda\vdash m+n-k+1} d_{\lambda \gamma}^{(n)} \widetilde{H}_\lambda[MB_\beta] \\
		\text{(using \eqref{eq:rel_cmunu_dmunu})}& = \sum_{\lambda\vdash m+n-k+1} \frac{\widetilde{H}_\lambda[MB_\beta]}{w_\lambda} \sum_{\gamma \vdash m-k+1}  c_{\lambda \gamma}^{(n)}  B_\gamma T_\gamma \\
		\text{(using \eqref{eq:lem_eBcmunu})}& = \sum_{\lambda\vdash m+n-k+1} \frac{\widetilde{H}_\lambda[MB_\beta]}{w_\lambda} e_{m-k}[B_\lambda-1]B_\lambda.
	\end{align*}
	Using $e_{m-k}[B_\lambda]=e_{m-k}[B_\lambda-1]+e_{m-k-1}[B_\lambda-1]$, in order to prove \eqref{eq:auxiliar}, it remains to show that 
	\[\sum_{\lambda\vdash m+n-k+1} \frac{\widetilde{H}_\lambda[MB_\beta]}{w_\lambda} e_{m-k-1}[B_\lambda-1]B_\lambda = 0,\]
	which is clear by following our last computation backward for $e_{m-k-1}[B_\lambda-1]$:
	\begin{align*}
		& \hspace{-0.5cm}\sum_{\lambda\vdash m+n-k+1} \frac{\widetilde{H}_\lambda[MB_\beta]}{w_\lambda} e_{m-k-1}[B_\lambda-1]B_\lambda=\\
		\text{(using \eqref{eq:lem_eBcmunu})}& = \sum_{\lambda\vdash m+n-k+1} \frac{\widetilde{H}_\lambda[MB_\beta]}{w_\lambda} \sum_{\gamma \vdash m-k+1}  c_{\lambda \gamma}^{(n+1)}  B_\gamma T_\gamma \\
		\text{(using \eqref{eq:rel_cmunu_dmunu})}& = \sum_{\gamma \vdash m-k+1} \frac{B_\gamma T_\gamma}{w_\gamma} \sum_{\lambda\vdash m+n-k+1} d_{\lambda \gamma}^{(n+1)} \widetilde{H}_\lambda[MB_\beta] \\
		\text{(using \eqref{eq:def_dmunu})}& = \sum_{\gamma \vdash m-k+1} \frac{B_\gamma T_\gamma}{w_\gamma} e_{n+1}[B_\beta] \widetilde{H}_\gamma[MB_\beta]\\
		& =0
	\end{align*}
	as $e_{n+1}[B_\beta]=0$ (since $\beta \vdash n$). This completes the proof of the theorem.
\end{proof}
\begin{corollary}
	\label{cor:Deltahh_ehh}
	For $n,m,k\geq 0$ and $m\geq k\geq 0$, we have
	\begin{equation}
		\< \Delta_{e_{m+n-k-1}}'e_{m+n}, h_{m}h_n \>= \< \nabla e_{m+n-k}, e_k h_{n-k} h_{m-k} \> .
	\end{equation}
\end{corollary}
\begin{proof}
	Just combine \eqref{eq:Deltahh_lem} and Theorem~\ref{thm:new_id}.
\end{proof}
\begin{corollary}
	\label{cor:ehh}
	For $n,m,k\geq 0$ and $m\geq k\geq 0$, we have
	\begin{equation}
		\sum_{r=1}^{m-k+1} t^{m-k-r+1} \< \Delta_{h_{m-k-r+1}} \Delta_{e_k} e_n \left[ X \dfrac{1 - q^r}{1 - q} \right], e_n \>= \< \nabla e_{m+n-k}, e_k h_{n-k} h_{m-k} \> .
	\end{equation}
\end{corollary}
\begin{proof}
	Just combine Theorem~\ref{thm:delta_hh} and Theorem~\ref{thm:new_id}.
\end{proof}

\section{The newdinv}

In this section we show that the newdinv is actually the natural dinv statistic on a specific subset of partially labelled Dyck paths.

\subsection{Definition of the newdinv}

The definition of the newdinv depends on a bijection $\Phi$ defined in \cite{Duane-Garsia-Zabrocki-2013}*{Section~4}. First of all, we consider the two shuffle parking functions as two car parking functions (starting with a $2$), and we think of these paths as sequences of \emph{dominoes} $[\ell_i, a_i]$ where $\ell_i$ is the label in the $i$-th row, and $a_i$ is the $i$-th letter in the area word. Then $\Phi$ acts as follows: consider the subsequence of dominoes between the first one (included; it is necessarily a $[2,0]$) and the next $[2,0]$ (excluded, if any; otherwise consider the whole sequence). If the subsequence consists of the domino $[2,0]$ only, then it is removed and $\Phi$ does not do anything else. Otherwise, we remove the next domino (which must be a $[1,0]$), then we replace all the $[2,a]$ in the subsequence with $[2, a\!-\!1]$, and finally we replace all the pairs of dominoes $[1,a] [2,a\!-\!1]$ with $[1,a\!-\!1] [2,a]$ if there is any. They define the newdinv recursively as $\newdinv(PF) \coloneqq k-1 + \newdinv(\Phi(PF))$, where $k$ is the number of $[2,0]$ dominoes in the parking function.

\subsection{Two bijections involving polyominoes} \label{sec:two_bij}


In \cite{DAdderio-Iraci-polyominoes-2017}*{Theorem~8.4} the authors define a bijection; a slight modification of this function bijectively maps reduced polyominoes of size $m \times n$ to the ``Catalan'' partially labelled Dyck paths, i.e. paths with $m+1$ zero valleys, $n$ positive labels (that are increasing in the dinv reading word order), and $n$ decorated rises (which forces a zero label to be in the bottom-left corner, as in Definition~\ref{def:zerocomp}). Notice that every step with a non-zero label must be a decorated rise.

We describe here the inverse map, which we call $\eta^{-1}$, as it is easier. See Figure~\ref{fig:plbounce} for an example. It works as follows: looking at the rows of the partially labelled Dyck path going bottom to top, we draw a horizontal red step if the row contains a zero valley, and we draw a vertical step otherwise; next we draw a horizontal green step below each horizontal red step such that the number of squares between them is the same as the number of squares in the row containing the zero valley corresponding to the matching horizontal red step; finally, we draw the vertical green steps in the only possible way. The image is the set of $m \times n$ reduced polyominoes (the first two steps are the ghost steps), and the area is trivially preserved, as the lengths of the rows that contribute to the area are the same as the heights of the columns of the polyomino.

\begin{figure}[!ht]
	\begin{center}
		\begin{tikzpicture}[scale=0.5]
		\draw[step=1.0, gray!60, thin] (0,0) grid (12,12);
		
		\draw[gray!60, thin] (0,0) -- (12,12);
		
		\draw[blue!60, line width=1.6pt] (0,0) -- (0,1) -- (0,2) -- (0,3) -- (1,3) -- (1,4) -- (2,4) -- (2,5) -- (3,5) -- (4,5) -- (4,6) -- (4,7) -- (4,8) -- (5,8) -- (6,8) -- (6,9) -- (6,10) -- (7,10) -- (7,11) -- (8,11) -- (8,12) -- (9,12) -- (10,12) -- (11,12) -- (12,12);
		
		
		
		\node at (-0.5,1.5) {$*$};
		\node at (-0.5,2.5) {$*$};
		\node at (3.5,6.5) {$*$};
		\node at (3.5,7.5) {$*$};
		\node at (5.5,9.5) {$*$};
		
		\draw
		(0.5,0.5) circle (0.4 cm) node {$0$}
		(0.5,1.5) circle (0.4 cm) node {$1$}
		(0.5,2.5) circle (0.4 cm) node {$2$}
		(1.5,3.5) circle (0.4 cm) node {$0$}
		(2.5,4.5) circle (0.4 cm) node {$0$}
		(4.5,5.5) circle (0.4 cm) node {$0$}
		(4.5,6.5) circle (0.4 cm) node {$3$}
		(4.5,7.5) circle (0.4 cm) node {$4$}
		(6.5,8.5) circle (0.4 cm) node {$0$}
		(6.5,9.5) circle (0.4 cm) node {$5$}
		(7.5,10.5) circle (0.4 cm) node {$0$}
		(8.5,11.5) circle (0.4 cm) node {$0$};
		\end{tikzpicture}
		\begin{tikzpicture}[scale=0.8]
		\draw[step=1.0, gray!60, thin] (0,0) grid (6,5);
		
		\filldraw[yellow, opacity=0.3] (0,0) -- (1,0) -- (2,0) -- (2,1) -- (3,1) -- (3,2) -- (4,2) -- (5,2) -- (6,2) -- (6,3) -- (6,4) -- (6,5) -- (5,5) --  (4,5) -- (4,4) -- (3,4) -- (3,3) -- (3,2) -- (2,2) -- (1,2) -- (0,2) -- (0,1) -- (0,0);
		
		\draw[red, sharp <-sharp >, sharp > angle = -45, line width=1.6pt] (-1,0) -- (0,0) -- (0,1) -- (0,2) -- (1,2) -- (2,2) -- (3,2) -- (3,3) -- (3,4) -- (4,4) -- (4,5) -- (5,5) -- (6,5);
		
		\draw[green, sharp <-sharp >, sharp > angle = 45, line width=1.6pt] (-1,0) -- (0,0) -- (1,0) -- (2,0) -- (2,1) -- (3,1) -- (3,2) -- (4,2) -- (5,2) -- (6,2) -- (6,3) -- (6,4) -- (6,5);
		
		\draw[red, dashed, sharp <-sharp >, sharp > angle = -45, line width=1.6pt] (-1,0) -- (0,0) -- (0,1) -- (0,2) -- (1,2) -- (2,2) -- (3,2) -- (3,3) -- (3,4) -- (4,4) -- (4,5) -- (5,5) -- (6,5);
		
		\draw[red, sharp <-sharp >, sharp > angle = -45, line width=1.6pt] (2,2) -- (3,2) -- (3,3);
		
		\node[white] at (-2,-0.8) {};
		\end{tikzpicture}
	\end{center}
	
	\caption{A partially labelled Dyck path in our special subset (left), and the corresponding polyomino (right).}
	\label{fig:plbounce}
\end{figure}

In \cite{DAdderio-Iraci-polyominoes-2017}*{Theorem~4.5} the authors define another bijection between $\RP(m,n)^{\ast k}$ and $\PF^2(m,n)^{\ast k}$, preserving both dinv and area. We take a slightly modified version, which we call $\psi$. See Figure~\ref{fig:2cpf-bijection} for an example. It consists of taking the area word of the polyomino, building a Dyck path with the same area word (disregarding bars), putting $1$'s in rows corresponding to barred letters and $2$'s in rows corresponding to unbarred letter, and keeping the decorations on the rises exactly as they were. The image is the set of two car parking functions with $m+1$ $2$'s and $n$ $1$'s that start with a $2$, where the starting $2$ plays the same role as the ghost steps in the polyomino, since adding (or removing) a $2$ in the first row does not alter the statistics in any way. Removing it we get an element in $\PF^2(m,n)^{\ast k}$, but it is more convenient to just keep it.

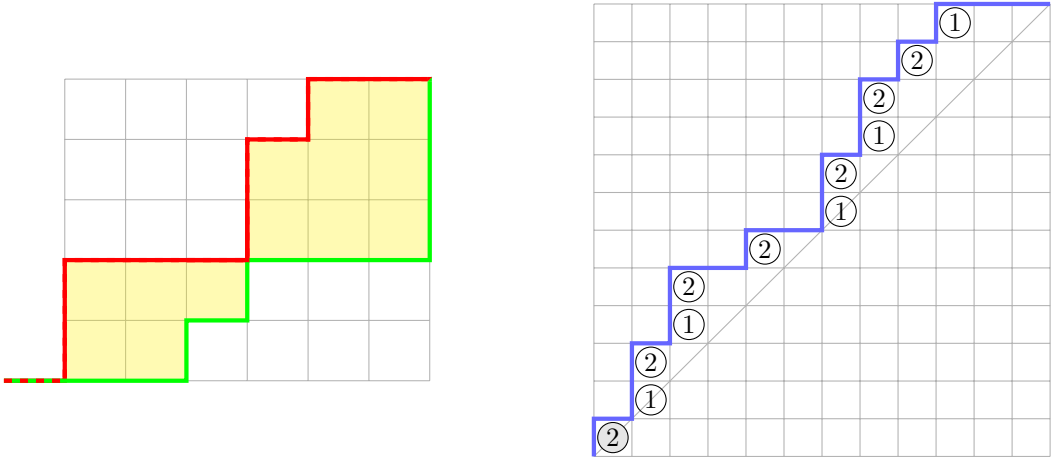
\begin{figure}[!ht]
	\begin{minipage}{0.5\textwidth}
		\begin{center}
			\begin{tikzpicture}[scale = 0.8]
			\draw[step=1.0, gray, opacity=0.6,thin] (0,0) grid (6,5);
			
			\filldraw[yellow, opacity=0.3] (0,0) -- (1,0) -- (2,0) -- (2,1) -- (3,1) -- (3,2) -- (4,2) -- (5,2) -- (6,2) -- (6,3) -- (6,4) -- (6,5) -- (5,5) -- (4,5) -- (4,4) -- (3,4) -- (3,3) -- (3,2) -- (2,2) -- (1,2) -- (0,2) -- (0,1) -- (0,0);
			
			\draw[red, sharp <-sharp >, sharp > angle = -45, line width=1.6pt] (-1,0) -- (0,0) -- (0,1) -- (0,2) -- (1,2) -- (2,2) -- (3,2) -- (3,3) -- (3,4) -- (4,4) -- (4,5) -- (5,5) -- (6,5);
			
			\draw[green, sharp <-sharp >, sharp > angle = 45, line width=1.6pt] (-1,0) -- (0,0) -- (1,0) -- (2,0) -- (2,1) -- (3,1) -- (3,2) -- (4,2) -- (5,2) -- (6,2) -- (6,3) -- (6,4) -- (6,5);
			
			\draw[red, dashed, sharp <-sharp >, sharp > angle = -45, line width=1.6pt] (-1,0) -- (0,0) -- (0,1) -- (0,2) -- (1,2) -- (2,2) -- (3,2) -- (3,3) -- (3,4) -- (4,4) -- (4,5) -- (5,5) -- (6,5);
			
			\draw[red, sharp <-sharp >, sharp > angle = -45, line width=1.6pt] (2,2) -- (3,2) -- (3,3);
			\end{tikzpicture}
		\end{center}
	\end{minipage}%
	\begin{minipage}{0.5\textwidth}
		\begin{center}
			\begin{tikzpicture}[scale = 0.5]
			\draw[gray!60, thin] (-1,-1) grid (11,11);
			\draw[gray!60, thin] (-1,-1) -- (11,11);
			
			\draw[blue!60, line width=1.6pt] (-1,-1) -- (-1,0) -- (0,0) -- (0,1) -- (0,2) -- (1,2) -- (1,3) -- (1,4) -- (2,4) -- (3,4) -- (3,5) -- (4,5) -- (5,5) -- (5,6) -- (5,7) -- (6,7) -- (6,8) -- (6,9) -- (7,9) -- (7,10) -- (8,10) -- (8,11) -- (9,11) -- (10,11) -- (11,11);

			\filldraw[fill=gray!20]
			(-0.5,-0.5) circle (0.4 cm) node {$2$};
			
			\draw
			(0.5,0.5) circle (0.4 cm) node {$1$}
			(0.5,1.5) circle (0.4 cm) node {$2$}
			(1.5,2.5) circle (0.4 cm) node {$1$}
			(1.5,3.5) circle (0.4 cm) node {$2$}
			(3.5,4.5) circle (0.4 cm) node {$2$}
			(5.5,5.5) circle (0.4 cm) node {$1$}
			(5.5,6.5) circle (0.4 cm) node {$2$}
			(6.5,7.5) circle (0.4 cm) node {$1$}
			(6.5,8.5) circle (0.4 cm) node {$2$}
			(7.5,9.5) circle (0.4 cm) node {$2$}
			(8.5,10.5) circle (0.4 cm) node {$1$};
			\end{tikzpicture}
		\end{center}
	\end{minipage}
	\caption{A reduced polyomino (left) with area word $0 \bar{0} 1 \bar{1} 2 1 0 \bar{0} 1 \bar{1} 2 2 \bar{2}$ and its corresponding two car parking function (right).}
	\label{fig:2cpf-bijection}
\end{figure}

We omit a proof of the bijectivity of the maps $\psi$ and $\eta^{-1}$ as they are analogous to the ones in \cite{DAdderio-Iraci-polyominoes-2017}.

Notice that both these maps actually have stronger properties, which are described in \cite{DAdderio-Iraci-polyominoes-2017}.

\subsection{Relation with the newdinv}

It turns out that the composition $\psi \circ \eta^{-1}$ of the bijections described above preserves the area statistic, and sends the natural dinv statistic into the newdinv. 

In fact, it translates the recursion of \cite{Duane-Garsia-Zabrocki-2013}*{Theorem~1.1} in the usual recursion for labelled Dyck paths that occurred for example in \cite{Zabrocki-4Catalan-2016} and \cite{Carlsson-Mellit-ShuffleConj-2015}. Indeed, if we let $\Phi$ to be the map defined in \cite{Duane-Garsia-Zabrocki-2013}*{Section~4}, then the composition $\eta \circ \psi^{-1} \circ \Phi \circ \psi \circ \eta^{-1}$ is essentially the same recursive step used for example in \cite{Zabrocki-4Catalan-2016}*{Proposition~4}. This fact is the core of the content of the proof of the following theorem.

\begin{theorem}
	The composition $\psi \circ \eta^{-1}$ of the two bijections maps the natural dinv statistic on partially labelled Dyck paths to the newdinv statistic on the two shuffle parking functions, and preserves the area.
	
	It also maps the zero composition on the partially labelled Dyck paths to the big car composition on the two shuffle parking functions.
\end{theorem}

\begin{proof}
	In order to prove that the dinv is mapped to the newdinv, it is enough to show that the map $\eta \circ \psi^{-1} \circ \Phi \circ \psi \circ \eta^{-1}$ decreases the dinv by $k-1$ when applied to a partially labelled Dyck path as above, where $k$ is the number of times that the path touches the main diagonal (not counting the endpoint - notice that $k$ must also be the number of zero valleys on the main diagonal).
	
	Let us first see what $\psi^{-1} \circ \Phi \circ \psi$ does. In terms of area word, it is extremely easy: it takes the area word of the polyomino and it removes the initial $0$ if it starts with two consecutive $0$'s; otherwise, it takes the subsequence of the area word strictly before the second $0$, removes the first $\bar{0}$ (there must be one), decreases all the unbarred letters by one unit, changes all the sequences $a\!-\!1 \; \bar{a}$ to $\overline{a\!-\!1} \; a$, if any, and then moves the modified subsequence to the end of the area word.
	
	Now the goal is to understand what happens to the picture. If the polyomino starts with two consecutive horizontal red steps (the first being the ghost one), then it just deletes one. If it doesn't, however, we claim that what happens is that the first vertical red step is removed, the last vertical green step before the second pair of overlapping horizontal steps is removed (if any; otherwise, the last vertical green step is removed - notice that the first pair of overlapping horizontal steps consists of the first two steps), and then the portion of the two paths between the first pair of overlapping steps (included) and the last one (excluded) is cycled to the end. See Figure~\ref{fig:phi} for an example.
	
	\begin{figure}[!ht]
		\begin{minipage}{0.5\textwidth}
			\begin{center}
				\begin{tikzpicture}[scale = 0.66]
				\draw[gray!60, thin] (0,0) grid (6,11);
				
				\filldraw[yellow, opacity=0.3] (0,0) -- (1,0) -- (2,0) -- (2,1) -- (2,2) -- (2,3) -- (2,4) -- (3,4) -- (3,5) -- (4,5) -- (4,6) -- (4,7) -- (4,8) -- (5,8) -- (6,8) -- (6,9) -- (6,10) -- (6,11) -- (5,11) -- (5,10) -- (5,9) -- (4,9) -- (4,8) -- (4,7) -- (4,6) -- (3,6) -- (3,5) -- (3,4) -- (2,4) -- (1,4) -- (1,3) -- (1,2) -- (0,2) -- (0,1) -- (0,0);
				
				\draw[red, sharp <-sharp >, sharp > angle = -45, line width=1.6pt] (-1,0) -- (0,0) -- (0,1) -- (0,2) -- (1,2) -- (1,3) -- (1,4) -- (2,4) -- (3,4) -- (3,5) -- (3,6) -- (4,6) -- (4,7) -- (4,8) -- (4,9) -- (5,9) -- (5,10) -- (5,11) -- (6,11);
				
				\draw[green, sharp <-sharp >, sharp > angle = 45, line width=1.6pt] (-1,0) -- (0,0) -- (1,0) -- (2,0) -- (2,1) -- (2,2) -- (2,3) -- (2,4) -- (3,4) -- (3,5) -- (4,5) -- (4,6) -- (4,7) -- (4,8) -- (5,8) -- (6,8) -- (6,9) -- (6,10) -- (6,11);
				
				\draw[red, dashed, sharp <-sharp >, sharp > angle = -45, line width=1.6pt] (-1,0) -- (0,0) -- (0,1) -- (0,2) -- (1,2) -- (1,3) -- (1,4) -- (2,4) -- (3,4) -- (3,5) -- (3,6) -- (4,6) -- (4,7) -- (4,8) -- (4,9) -- (5,9) -- (5,10) -- (5,11) -- (6,11);
				
				\draw[orange, line width = 1.6 pt] (-1.1,-0.1) -- (2.1,-0.1) -- (2.1,4.1) -- (-1.1,4.1) -- cycle;
				\end{tikzpicture}
			\end{center}
		\end{minipage}%
		\begin{minipage}{0.5\textwidth}
			\begin{center}
				\begin{tikzpicture}[scale = 0.66]
				\draw[gray!60, thin] (0,0) grid (6,10);
				
				\filldraw[yellow, opacity=0.3] (0,0) -- (0,1) -- (0,2) -- (1,2) -- (1,3) -- (1,4) -- (1,5) -- (2,5) -- (2,6) -- (2,7) -- (3,7) -- (4,7) -- (4,8) -- (5,8) -- (5,9) -- (5,10) -- (6,10) -- (6,9) -- (6,8) -- (6,7) -- (5,7) -- (4,7) -- (3,7) -- (3,6) -- (3,5) -- (3,4) -- (2,4) -- (1,4) -- (1,3) -- (1,2) -- (1,1) -- (0,1) -- (0,0);
				
				\draw[red, sharp <-sharp >, sharp > angle = -45, line width=1.6pt] (-1,0) -- (0,0) -- (0,1) -- (0,2) -- (1,2) -- (1,3) -- (1,4) -- (1,5) -- (2,5) -- (2,6) -- (2,7) -- (3,7) -- (4,7) -- (4,8) -- (5,8) -- (5,9) -- (5,10) -- (6,10);
				
				\draw[green, sharp <-sharp >, sharp > angle = 45, line width=1.6pt] (-1,0) -- (0,0) -- (0,1) -- (1,1) -- (1,2) -- (1,3) -- (1,4) -- (2,4) -- (3,4) -- (3,5) -- (3,6) -- (3,7) -- (4,7) -- (5,7) -- (6,7) -- (6,8) -- (6,9) -- (6,10);
				
				\draw[red, dashed, sharp <-sharp >, sharp > angle = -45, line width=1.6pt] (-1,0) -- (0,0) -- (0,1) -- (0,2) -- (1,2) -- (1,3) -- (1,4) -- (1,5) -- (2,5) -- (2,6) -- (2,7) -- (3,7) -- (4,7) -- (4,8) -- (5,8) -- (5,9) -- (5,10) -- (6,10);
				
				\draw[orange, line width = 1.6 pt] (2.9,6.9) -- (6.1,6.9) -- (6.1,10.1) -- (2.9,10.1) -- cycle;
				\end{tikzpicture}
			\end{center}
		\end{minipage}
		\caption{The reduced polyomino in Figure~\ref{fig:aw-reduced-polyomino} (left) and its image via the map $\psi^{-1} \circ \Phi \circ \psi$ (right). The section before the second pair of overlapping steps (boxed in orange) is shrunk vertically by one unit and moved to the end.}
		\label{fig:phi}
	\end{figure}
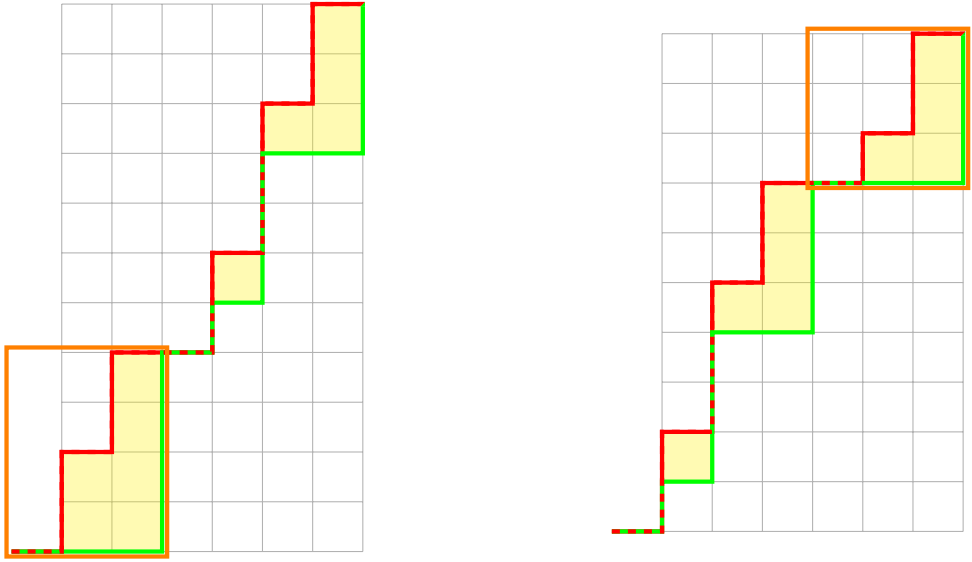
	
	The first thing to notice is that the $0$'s in the area word correspond to the pairs of overlapping horizontal steps. This is clear from the construction of the area word in \cite{Aval-DAdderio-Dukes-Hicks-LeBorgne-2014}*{Section~2} (see also the first algorithm described in \cite{DAdderio-Iraci-polyominoes-2017}*{Subsection~2.1}), and it means that the region involved in this process is the one between the first two pairs of overlapping horizontal steps. Now, the first vertical red step corresponds to the first $\bar{0}$, which is deleted as expected.
	
	Deleting the last vertical green step in the region means that the two paths are vertically pushed one step closer to each other. When doing this, the length of the diagonal of slope $-1$ starting from the end of a horizontal green step (which determine the values of the unbarred letters in the area word) decreases by one if and only if the other end hits the beginning of a horizontal red step, which is equivalent to saying that the matching letter is not part of a rise in the reading word: rises correspond exactly to diagonals of slope $-1$ going from the end of a horizontal green step to the beginning of a vertical red step, as it is clear from a correspondence with Dyck paths shown in \cite{Aval-DAdderio-Dukes-Hicks-LeBorgne-2014}*{Section~3} (see also the second algorithm described in \cite{DAdderio-Iraci-polyominoes-2017}*{Subsection~2.1}). This means that the corresponding unbarred letters decrease by one unit.
	
	For the same reason the value of the barred letters that are part of a rise decreases by one unit (this is clear from the correspondence with Dyck paths in \cite{Aval-DAdderio-Dukes-Hicks-LeBorgne-2014}*{Section~3}, as the vertical red step is sliding below the matching horizontal green step); notice also that the red step, being pushed down, is now read \emph{before} the horizontal green step which was part of the same rise. This means that a rise $a \; \bar{a}$ is changed into $\overline{a\!-\!1} \; a$, which is the same as decreasing the unbarred letter one unit and then change the sequences $a\!-\!1 \; \bar{a}$ we get to $\overline{a\!-\!1} \; a$.  It is immediate that barred letters that were not part of a rise keep their position in the area word instead. Cycling the region to the end simply does the same to the area word.
	
	This proves that $\psi^{-1} \circ \Phi \circ \psi$ acts on reduced polyominoes as we claimed. Now we have to check what happens to partially labelled Dyck paths. If a polyomino starts with two horizontal red steps then the corresponding Dyck path starts with two consecutive zero valleys on the diagonal, and $\eta \circ \psi^{-1} \circ \Phi \circ \psi \circ \eta^{-1}$ simply deletes one and doesn't change the dinv. Otherwise, first of all notice that $0$'s in the area word of the polyomino correspond to zero valleys on the diagonal (it is clear because they correspond to overlapping horizontal steps). Now, deleting the first vertical red step corresponds to deleting the first vertical step that is not a valley, which is the one immediately above the starting zero valley (hence it lies in the second row; since the red path is not otherwise altered, the sequence of zeroes and positive labels remains the same. Next, deleting the last vertical green step decreases the number of squares in each column of the interested region by one, which means that the zero valleys in the matching region of the Dyck path (except the one in the bottom left corner) should get one step closer to the main diagonal, and since the relative positions of zeroes and positive labels does not change, this maneuver corresponds to deleting the last horizontal step before the next zero valley on the diagonal (if any; otherwise it is the last horizontal step). Finally, the whole region is cycled to the end. This maneuver doesn't change the dinv (the primary becomes secondary and viceversa), except for the contribution of the deleted step: its primary dinv is balanced by the secondary dinv gained by the zero valley that lied in the bottom left corner, but its secondary dinv is lost, and it formed secondary inversions with the zero valleys on the diagonal (except the bottom left one), which are exactly $k-1$. 
	
	This proves that $\psi \circ \eta^{-1}$ maps the dinv to the newdinv. Checking that the compositions match is trivial, as both match the composition of the unbarred letters of the polyomino, where the breakpoints of the parts are the $0$'s.
\end{proof}

\section{The ``$ehh$'' case of the shuffle conjecture}

In this section we prove the combinatorial counterpart of Corollary~\ref{cor:Deltahh_ehh}, which we recall here: for $n,m,k\geq 0$ and $n, m\geq k\geq 0$, we have
\begin{equation}
	\< \Delta_{e_{m+n-k-1}}'e_{m+n}, h_{m}h_n \>  =\< \nabla e_{m+n-k}, e_k h_{n-k} h_{m-k} \>.
\end{equation}

\subsection{The ``$ehh$'' bijection}

Let us call (with a slight abuse of notation) \emph{$(k,n,m)$-shuffle paths} the elements of $\PF(m+n-k)$ whose dinv reading word lies in the shuffle \[ (1, \dots, k) \shuffle (n, \dots, k+1) \shuffle (m+n-k, \dots, n+1). \]

\begin{theorem} \label{thm:ehh_bij}
	There is a bijective correspondence between $\PF^2(m,n)^{\ast k}$ and the set of $(k,n,m)$-shuffle paths that preserves the bistatistic $(\dinv, \area)$.
\end{theorem}
\begin{remark}
	This result does not provide any new evidence by itself (since the shuffle conjecture has been completely proved), but it suggests that a similar approach could be used to solve the $e_a h_b h_c$ case of the Delta conjecture as well, since it implicitly provides a combinatorial recursion similar to those that are used to prove the $e_a h_b$ and $h_a h_b$ cases (cf. \cite{DAdderio-Iraci-VandenWyngaerd-GenDeltaSchroeder} and \cite{DAdderio-Iraci-polyominoes-2017}). 
\end{remark}
\begin{proof}
	Let $D$ be a $(k,n,m)$-shuffle path. First of all, replace all the numbers from $k+1$ to $n$ with bold $\textbf{1}$'s and all the numbers from $n+1$ to $m+n-k$ with bold $\textbf{2}$'s, with the convention that any bold number is bigger than any regular one. It is immediate to check that this operation does not alter the dinv (since it preserves inversions).
	
	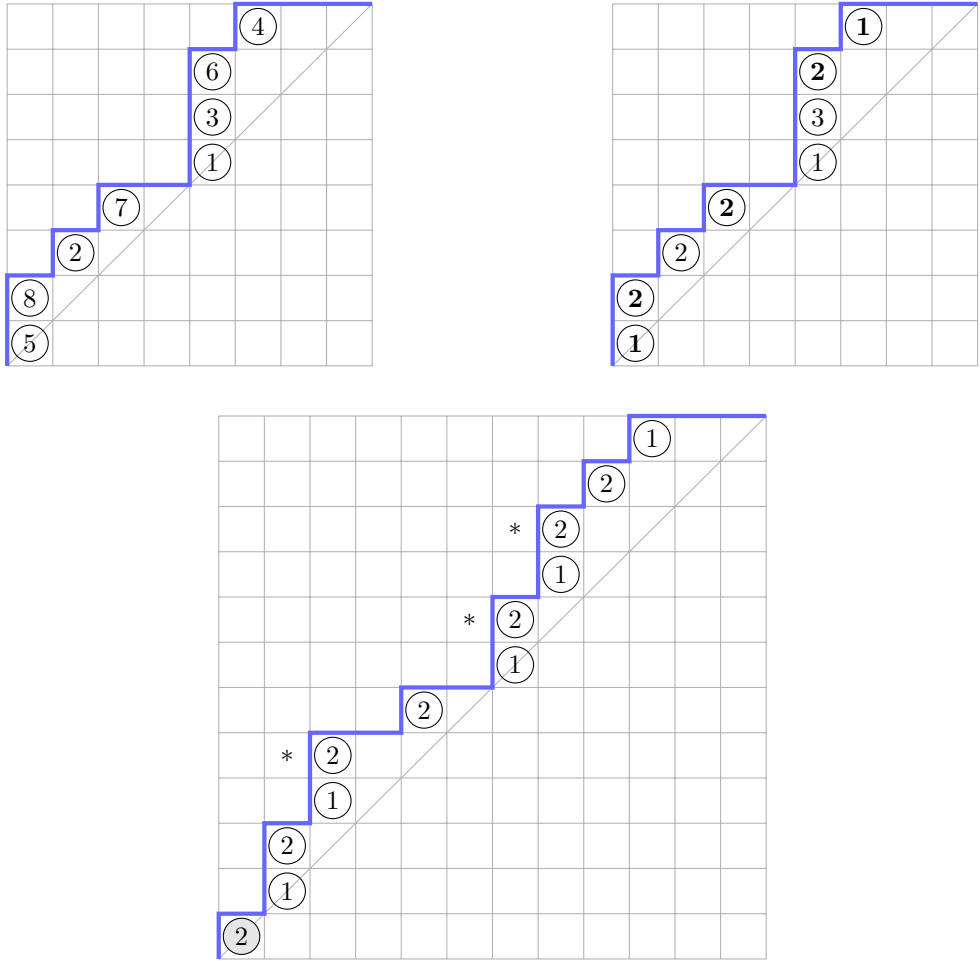
\begin{figure}[!ht]
		\centering
		\begin{minipage}{0.5 \textwidth}
			\centering
			\begin{tikzpicture}[scale = 0.6]
			\draw[gray!60, thin] (0,0) grid (8,8);
			\draw[gray!60, thin] (0,0) -- (8,8);
			
			\draw[blue!60, line width=1.6pt] (0,0) -- (0,1) -- (0,2) -- (1,2) -- (1,3) -- (2,3) -- (2,4) -- (3,4) -- (4,4) -- (4,5) -- (4,6) -- (4,7) -- (5,7) -- (5,8) -- (6,8) -- (7,8) -- (8,8);
			
			\draw
			(0.5,0.5) circle (0.4 cm) node {$5$}
			(0.5,1.5) circle (0.4 cm) node {$8$}
			(1.5,2.5) circle (0.4 cm) node {$2$}
			(2.5,3.5) circle (0.4 cm) node {$7$}
			(4.5,4.5) circle (0.4 cm) node {$1$}
			(4.5,5.5) circle (0.4 cm) node {$3$}
			(4.5,6.5) circle (0.4 cm) node {$6$}
			(5.5,7.5) circle (0.4 cm) node {$4$};
			\end{tikzpicture}
		\end{minipage}%
		\begin{minipage}{0.5 \textwidth}
			\centering
			\begin{tikzpicture}[scale = 0.6]
			\draw[gray!60, thin] (0,0) grid (8,8);
			\draw[gray!60, thin] (0,0) -- (8,8);
			
			\draw[blue!60, line width=1.6pt] (0,0) -- (0,1) -- (0,2) -- (1,2) -- (1,3) -- (2,3) -- (2,4) -- (3,4) -- (4,4) -- (4,5) -- (4,6) -- (4,7) -- (5,7) -- (5,8) -- (6,8) -- (7,8) -- (8,8);
			
			\draw
			(1.5,2.5) circle (0.4 cm) node {$2$}
			(4.5,4.5) circle (0.4 cm) node {$1$}
			(4.5,5.5) circle (0.4 cm) node {$3$};
			
			\draw
			(0.5,0.5) circle (0.4 cm) node {$\textbf{1}$}
			(0.5,1.5) circle (0.4 cm) node {$\textbf{2}$}
			(2.5,3.5) circle (0.4 cm) node {$\textbf{2}$}
			(4.5,6.5) circle (0.4 cm) node {$\textbf{2}$}
			(5.5,7.5) circle (0.4 cm) node {$\textbf{1}$};
			\end{tikzpicture}
		\end{minipage}
		\begin{minipage}{0.8 \textwidth}
			\centering
			\begin{tikzpicture}[scale = 0.6]
			\draw[draw=none, use as bounding box] (-2,-2) rectangle (12,12);
			\draw[gray!60, thin] (-1,-1) grid (11,11);
			\draw[gray!60, thin] (-1,-1) -- (11,11);
			
			\draw[blue!60, line width=1.6pt] (-1,-1) -- (-1,0) -- (0,0) -- (0,1) -- (0,2) -- (1,2) -- (1,3) -- (1,4) -- (2,4) -- (3,4) -- (3,5) -- (4,5) -- (5,5) -- (5,6) -- (5,7) -- (6,7) -- (6,8) -- (6,9) -- (7,9) -- (7,10) -- (8,10) -- (8,11) -- (9,11) -- (10,11) -- (11,11);

			\filldraw[fill=gray!20]
			(-0.5,-0.5) circle (0.4 cm) node {$2$};
			
			\draw
			(0.5,0.5) circle (0.4 cm) node {$1$}
			(0.5,1.5) circle (0.4 cm) node {$2$}
			(1.5,2.5) circle (0.4 cm) node {$1$}
			(1.5,3.5) circle (0.4 cm) node {$2$}
			(3.5,4.5) circle (0.4 cm) node {$2$}
			(5.5,5.5) circle (0.4 cm) node {$1$}
			(5.5,6.5) circle (0.4 cm) node {$2$}
			(6.5,7.5) circle (0.4 cm) node {$1$}
			(6.5,8.5) circle (0.4 cm) node {$2$}
			(7.5,9.5) circle (0.4 cm) node {$2$}
			(8.5,10.5) circle (0.4 cm) node {$1$};
			
			\draw
			(0.5,3.5) node {$\ast$}
			(4.5,6.5) node {$\ast$}
			(5.5,8.5) node {$\ast$};
			\end{tikzpicture}
		\end{minipage}%
		
		\caption{A parking function of size $8$ with dinv reading word $51827364 \in 123 \shuffle 54 \shuffle 876$, an intermediate step of the bijection, and the corresponding two car parking function.}
		\label{fig:2cpf-to-shuffle-ehh}
	\end{figure}
	Now, for any number $1 \leq i \leq k$, perform the following operation. First add a decorated vertical step immediately after the one with the $i$ label assigned, then add a horizontal step immediately after the new vertical step, next replace the label with a bold $\textbf{1}$, and assign a bold $\textbf{2}$ to the new decorated step. If this is done starting from $k$ and then going down to $1$, step by step both statistics are preserved: the area obviously does not change (we are adding a decorated letter to the area word without changing the other ones); the primary dinv on the left of $i$ is now primary dinv on the left of the new bold $\textbf{1}$; the primary dinv on the right of $i$ is now secondary dinv on the left of the new bold $\textbf{2}$; the secondary dinv on the left of $i$ is now either primary dinv on the left of the new bold $\textbf{2}$ (if it came from a bold $\textbf{1}$) or secondary dinv on the left of the new bold $\textbf{1}$ (if it came from a bold $\textbf{2}$); and the secondary dinv on the right of $i$ is now secondary dinv on the right of the new bold $\textbf{1}$. Finally, replace bold numbers with regular ones.
	
	The image is an element in $\PF^2(m,n)^{\ast k}$ with the same dinv and area of $D$. It remains to prove that the map is bijective, but it is since the inverse is easy to compute: given an element in $\PF^2(m,n)^{\ast k}$, replace all the $1$'s that are not immediately before a decorated rise (which are exactly $n-k$) with the numbers $n, \dots, k+1$ according to the order given by the dinv reading word (i.e. start with those on the main diagonal going bottom to top, then move to the diagonal $y=x+1$ going bottom to top, and so on); then do the same for all the $2$'s that are not decorated rises (there are exactly $m-k$ of them) with the numbers $m+n-k, \dots, n+1$; and finally delete the vertical steps containing decorated rises and the following horizontal steps (there must be one, since an element in $\PF^2(m,n)^{\ast k}$ can't have more than two consecutive vertical steps) and replace the labels immediately before the deleted rises with the numbers $1, \dots, k$. The ghost car has to be added at the end.
\end{proof}
\begin{remark}
	Combining Theorem~\ref{thm:ehh_bij}, Corollary~\ref{cor:Deltahh_ehh} and the proof of the ``$hh$'' case of the Delta conjecture in \cite{DAdderio-Iraci-polyominoes-2017}, we get a new, non-compositional proof of the ``$ehh$'' of the shuffle conjecture.
\end{remark}

In fact the recursion used in \cite{DAdderio-Iraci-polyominoes-2017} to prove the ``$hh$'' case of the Delta conjecture has a nice counterpart on the $(k,n,m)$-shuffle paths, which gives a combinatorial argument for the ``$ehh$'' case of the shuffle conjecture independent from the one in \cite{DAdderio-Iraci-polyominoes-2017}. This is the content of the next section.

\subsection{The ``$ehh$'' recursion}

From \cite{DAdderio-Iraci-polyominoes-2017} we have that $\PF^2_{q,t}(m \backslash r, n)^{\ast k}$ (where the parameter $r$ specifies the number of big cars on the main diagonal) satisfies the following recursion

\begin{align*}
	\PF^2_{q,t}(m \backslash r, n)^{\ast k} & = \sum_{s=1}^{n} \sum_{u=1}^{m-r+1} \sum_{h=0}^k t^{m+n-r-s-k+1} \qbinom{r+s-1}{s}_q \\ & \times q^{\binom{h}{2}} \qbinom{s}{h}_q \qbinom{s+u-h-1}{u-h}_q \PF^2_{q,t}(m-r \backslash u, n-s)^{\ast k-h}
\end{align*}

with initial conditions $\PF^2_{q,t}(m \backslash r, 0)^{\ast k} = \delta_{m,r} \delta_{k,0}$.

This recursion has an interpretation in terms of $(k,n,m)$-shuffle paths. Let us call \emph{small cars} the labels from $1$ to $k$, \emph{medium cars} the ones from $k+1$ to $n$, and \emph{big cars} the ones from $n+1$ to $m+n-k$. Notice that the small cars are the only ones in increasing order in the dinv reading word, so in particular there can be inversions involving only small cars. We have that $\PF^2_{q,t}(m \backslash r, n)^{\ast k}$ corresponds to the subset of $(k,n,m)$-shuffle paths of which $r-1$ big cars lie on the main diagonal. In the recursion, $h$ counts the small cars on the main diagonal, $s-h$ counts the medium cars on the main diagonal, $u-h$ counts the big cars at height $1$, and the recursive step is performed by deleting all the medium and big cars on the main diagonal, converting the small cars on the main diagonal to big cars, and then pushing everything else one step towards the diagonal (i.e. in the area word the $0$'s corresponding to medium and big cars are deleted, the $0$'s corresponding to small cars are converted to big cars, and all the other letters keep their car size and decrease by one their value). Next, since we necessarily get a labelled Dyck path with a big car in position $1$, we delete it and get a $(k-h,n-s,m-r)$-shuffle path with $u-1$ big cars on the main diagonal, completing the recursive step.

The proof of \cite{DAdderio-Iraci-polyominoes-2017}*{Theorem~6.1} shows that this same recursion with the same initial conditions is satisfied by $t^{m-k-r+1} \< \Delta_{h_{m-k-r+1}} \Delta_{e_k} e_n \left[ X \dfrac{1 - q^r}{1 - q} \right], e_n \>$. So this proves that the latter polynomial matches the $q,t$-enumerator of $(k,n,m)$-shuffle paths with $r-1$ big cars on the main diagonal. Summing over $r$, and using Corollary~\ref{cor:ehh}, we get a non-compositional proof of the ``$ehh$'' case of the shuffle conjecture, as claimed.

\section*{Acknowledgments}

The authors thank Anna Vanden Wyngaerd for many helpful comments and corrections.


\bibliographystyle{amsalpha}
\bibliography{Biblebib}

\end{document}